\documentclass[11pt]{article}
\usepackage[margin=1in]{geometry} 
\usepackage{amssymb,amsfonts,amsmath,bbm,mathrsfs,stmaryrd,mathtools}
\usepackage{xcolor}
\usepackage{url}

\usepackage[T3,T1]{fontenc}
\DeclareSymbolFont{tipa}{T3}{cmr}{m}{n}
\DeclareMathAccent{\invbreve}{\mathalpha}{tipa}{16}

\usepackage{enumerate}

\usepackage{hyperref}
\hypersetup{colorlinks,
             linkcolor=black!75!red,
             citecolor=blue,
             pdftitle={},
             pdfproducer={pdfLaTeX},
             pdfpagemode=None,
             bookmarksopen=true
             bookmarksnumbered=true}

\usepackage{tikz}
\usetikzlibrary{arrows,calc,decorations.pathreplacing,decorations.markings,intersections,shapes.geometric,through,fit,shapes.symbols,positioning,decorations.pathmorphing}

\usepackage{braket}

\usepackage[amsmath,thmmarks,hyperref]{ntheorem}
\usepackage{cleveref}

\creflabelformat{enumi}{#2(#1)#3}

\crefname{section}{Section}{Sections}
\crefformat{section}{#2Section~#1#3} 
\Crefformat{section}{#2Section~#1#3} 

\crefname{subsection}{\S}{\S\S}
\AtBeginDocument{%
  \crefformat{subsection}{#2\S#1#3}%
  \Crefformat{subsection}{#2\S#1#3}%
}

\crefname{subsubsection}{\S}{\S\S}
\AtBeginDocument{%
  \crefformat{subsubsection}{#2\S#1#3}%
  \Crefformat{subsubsection}{#2\S#1#3}%
}

\theoremstyle{plain}

\newtheorem{lemma}{Lemma}[section]
\newtheorem{proposition}[lemma]{Proposition}
\newtheorem{corollary}[lemma]{Corollary}
\newtheorem{theorem}[lemma]{Theorem}

\theoremstyle{nonumberplain}
\newtheorem{theoremN}{Theorem}

\theoremstyle{nonumberplain}
\theorembodyfont{\upshape}
\theoremsymbol{\ensuremath{\blacklozenge}}
\newtheorem{definitionN}{Definition}

\theoremstyle{plain}
\theorembodyfont{\upshape}
\theoremsymbol{\ensuremath{\blacklozenge}}

\newtheorem{definition}[lemma]{Definition}
\newtheorem{example}[lemma]{Example}
\newtheorem{examples}[lemma]{Examples}
\newtheorem{remark}[lemma]{Remark}

\newtheorem{construction}[lemma]{Construction}

\crefname{definition}{definition}{definitions}
\crefformat{definition}{#2definition~#1#3} 
\Crefformat{definition}{#2Definition~#1#3} 

\crefname{ex}{example}{examples}
\crefformat{example}{#2example~#1#3} 
\Crefformat{example}{#2Example~#1#3} 

\crefname{remark}{remark}{remarks}
\crefformat{remark}{#2remark~#1#3} 
\Crefformat{remark}{#2Remark~#1#3} 

\crefname{convention}{convention}{conventions}
\crefformat{convention}{#2convention~#1#3} 
\Crefformat{convention}{#2Convention~#1#3} 

\crefname{notation}{notation}{notations}
\crefformat{notation}{#2notation~#1#3} 
\Crefformat{notation}{#2Notation~#1#3} 

\crefname{table}{table}{tables}
\crefformat{table}{#2table~#1#3} 
\Crefformat{table}{#2Table~#1#3}

\crefname{lemma}{lemma}{lemmas}
\crefformat{lemma}{#2lemma~#1#3} 
\Crefformat{lemma}{#2Lemma~#1#3} 

\crefname{proposition}{proposition}{propositions}
\crefformat{proposition}{#2proposition~#1#3} 
\Crefformat{proposition}{#2Proposition~#1#3} 

\crefname{corollary}{corollary}{corollaries}
\crefformat{corollary}{#2corollary~#1#3} 
\Crefformat{corollary}{#2Corollary~#1#3} 

\crefname{theorem}{theorem}{theorems}
\crefformat{theorem}{#2theorem~#1#3} 
\Crefformat{theorem}{#2Theorem~#1#3} 

\crefname{enumi}{}{}
\crefformat{enumi}{(#2#1#3)}
\Crefformat{enumi}{(#2#1#3)}

\crefname{assumption}{assumption}{Assumptions}
\crefformat{assumption}{#2assumption~#1#3} 
\Crefformat{assumption}{#2Assumption~#1#3} 

\crefname{construction}{construction}{Constructions}
\crefformat{construction}{#2construction~#1#3} 
\Crefformat{construction}{#2Construction~#1#3} 

\crefname{equation}{}{}
\crefformat{equation}{(#2#1#3)} 
\Crefformat{equation}{(#2#1#3)}

\numberwithin{equation}{section}
\renewcommand{\theequation}{\thesection-\arabic{equation}}

\theoremstyle{nonumberplain}
\theoremsymbol{\ensuremath{\blacksquare}}

\newtheorem{proof}{Proof}

\newcommand\bA{{\mathbb A}}
\newcommand\bB{{\mathbb B}}
\newcommand\bC{{\mathbb C}}
\newcommand\bD{{\mathbb D}}
\newcommand\bE{{\mathbb E}}
\newcommand\bF{{\mathbb F}}
\newcommand\bG{{\mathbb G}}
\newcommand\bH{{\mathbb H}}

\newcommand\bK{{\mathbb K}}
\newcommand\bL{{\mathbb L}}
\newcommand\bM{{\mathbb M}}
\newcommand\bN{{\mathbb N}}

\newcommand\bP{{\mathbb P}}
\newcommand\bQ{{\mathbb Q}}
\newcommand\bR{{\mathbb R}}
\newcommand\bS{{\mathbb S}}
\newcommand\bT{{\mathbb T}}

\newcommand\bZ{{\mathbb Z}}

\newcommand\cH{{\mathcal H}}

\newcommand\cK{{\mathcal K}}

\newcommand\numberthis{\addtocounter{equation}{1}\tag{\theequation}}

\newcommand{\qedhere}{\mbox{}\hfill\ensuremath{\blacksquare}}

\title{Type-I permanence}
\author{Alexandru Chirvasitu}

\begin{document}

\date{}

\newcommand{\Addresses}{{
  \bigskip
  \footnotesize

  \textsc{Department of Mathematics, University at Buffalo, Buffalo,
    NY 14260-2900, USA}\par\nopagebreak \textit{E-mail address}:
  \texttt{achirvas@buffalo.edu}

}}

\maketitle

\begin{abstract}
  We prove a number of results on the survival of the type-I property under extensions of locally compact groups: (a) that given a closed normal embedding $\mathbb{N}\trianglelefteq\mathbb{E}$ of locally compact groups and a twisted action $(\alpha,\tau)$ thereof on a (post)liminal $C^*$-algebra $A$ the twisted crossed product $A\rtimes_{\alpha,\tau}\mathbb{E}$ is again (post)liminal and (b) a number of converses to the effect that under various conditions a normal, closed, cocompact subgroup $\mathbb{N}\trianglelefteq \mathbb{E}$ is type-I as soon as $\mathbb{E}$ is. This happens for instance if $\mathbb{N}$ is discrete and $\mathbb{E}$ is Lie, or if $\mathbb{N}$ is finitely-generated discrete (with no further restrictions except cocompactness). Examples show that there is not much scope for dropping these conditions.
  
  In the same spirit, call a locally compact group $\mathbb{G}$ type-I-preserving if all semidirect products $\mathbb{N}\rtimes \mathbb{G}$ are type-I as soon as $\mathbb{N}$ is, and {\it linearly} type-I-preserving if the same conclusion holds for semidirect products $V\rtimes\mathbb{G}$ arising from finite-dimensional $\mathbb{G}$-representations. We characterize the (linearly) type-I-preserving groups that are (1) discrete-by-compact-Lie, (2) nilpotent, or (3) solvable Lie.
\end{abstract}

\noindent {\em Key words: locally compact group; Lie group; type I; $C^*$-algebra; central extension; cohomology; exact sequence; LCA; nilpotent; solvable}

\vspace{.5cm}

\noindent{MSC 2020: 22D10; 22D05; 22D12; 22D15; 22D30; 22D35; 22D45; 22E25; 22E41; 46L05; 46L10}

\tableofcontents

\section*{Introduction}

The $C^*$-algebras (or locally compact groups) {\it of type I} are those for which it is possible to define a ``reasonable'' moduli space of irreducible $*$-representations on Hilbert spaces. The literature is by now staggering in depth and breadth, so we will content ourselves with citing some of the textbook-style sources (and pointing indirectly to {\it their} references): these include, for instance, \cite[Chapter 6]{ped-aut}, \cite[\S IV.1]{blk}, and \cite[Chapters 4 and 9]{dixc}.

Apart from `type-I' the phrases {\it GCR}, {\it postliminal} (or {\it postliminary}) and {\it having smooth dual} are also in use and synonymous, and we might revert to some of the other terms for variety. They were at various points introduced as separate properties, but are now known to all be equivalent under sensible separability assumptions \cite[Theorem 6.8.7]{ped-aut}. Many classes of (locally compact, second-countable) groups are known to be type-I:
\begin{itemize}
\item connected semisimple Lie groups;
\item connected nilpotent Lie groups;
\item connected components of real algebraic groups;
\item discrete groups precisely when they are {\it virtually abelian} (i.e. have finite-index abelian subgroups);
\item groups $\bG(k)$ of $k$-points for linear algebraic groups $\bG$ over local fields $k$ of characteristic zero.
\end{itemize}
Apart from the last item, which is \cite[Theorem 2]{be}, these are all recalled in \cite[Theorem 7.8]{folland} with citations to the original sources. 

One familiar technique for improving on such results is to assemble a group out of more manageable pieces, for which the property is already known to hold, and to show that that property survives under the various group-theoretic constructions involved: this is the {\it permanence} of the title.

Specifically, we are concerned here with how postliminality behaves under extensions
\begin{equation}\label{eq:introext}
  \{1\}\to \bN\to \bE\to \bG\to \{1\}
\end{equation}
of locally compact groups (always second-countable). Intuition and the examples dictate, for instance, that $\bE$ should again be of type I if $\bN$ is, provided $\bG$ is ``manageably small'':
\begin{itemize}
\item Thoma's characterization of type-I countable discrete groups \cite[Satz 6]{thoma} shows that finite-index embeddings make no difference.
\item Similarly, \cite[Lemme 3]{dix-alg} shows (by induction on the finite index $[\bE:\bN]$) that $\bE$ is type-I along with $\bN$ if the latter is normal and of finite index.
\item In fact $\bE$ is type-I whenever it contains a closed type-I subgroup $\bN\le \bE$ such that $\bE/\bN$ carries an $\bE$-invariant probability measure \cite[Theorem 1]{kal1} (see also \cite[Corollary 4.5]{gk}).
\item So in particular this certainly happens when $\bN\trianglelefteq \bE$ is normal and cocompact, as in that case the Haar probability measure on $\bE/\bN$ is $\bE$-invariant.
\end{itemize}

The upshot is that lifting (post)liminality along $\bN\le \bE$ goes through even upon relaxing the normality requirement. On the other hand, keeping $\bN\trianglelefteq\bE$ normal, one can add a {\it twisted action} \cite[\S 3]{qg-full} of $\bG$ into the mix, in the context of Green's {\it twisted covariance algebras} (\Cref{def:twisted}). In that setup a similar type-I-lifting result obtains (\Cref{th:tw}):

\begin{theoremN}
  Let $\bN\trianglelefteq\bE$ be a closed, cocompact normal subgroup of a second-countable locally compact group and $(A, \bE, \bN, \alpha,\tau)$ a twisted action on a separable $C^*$-algebra $A$.

  If $A$ is (post)liminal then so, respectively, is the twisted crossed product $A\rtimes_{\alpha,\tau}\bE$. \qedhere
\end{theoremN}

Deferring the details, roughly speaking, $A$ is to $A\rtimes_{\alpha,\tau}\bE$ as $\bN$ is to $\bE$ in the sense that the pair $(A,A\rtimes_{\alpha,\tau}\bE)$ specializes back to $(C^*(\bN),C^*(\bE))$ under suitable conditions (see \cite[p.199, Corollary]{gr-tw} and \Cref{cor:ext} below). In that sense, the theorem fits with the spirit of the present discussion. Furthermore, the duality-based proof of \Cref{th:tw} presumably allows for generalizations to {\it quantum} groups (for which imprimitivity/induction theory is available \cite{vs-imp}), though we will not pursue this here.

Given this wealth of material on (post)liminality permanence under passage to ``larger'' objects, two flavors of possible converse present themselves.

First, one might wonder whether postliminality {\it descends} from $\bE$ to $\bN$ when the latter is cocompact. This cannot happen in general, as attested by (non-type-I) lattices in (type-I) Lie groups \Cref{ex:latt}. Discrete {\it normal} cocompact subgroups will occasionally inherit the type-I property from the larger ambient group (\Cref{pr:cpctliequot} and \Cref{pr:fgok}):

\begin{theoremN}
  Let $\bN\trianglelefteq \bE$ be a discrete, cocompact, normal subgroup of a second-countable type-I locally compact group. $\bN$ is type-I if either
  \begin{itemize}
  \item $\bE$ (or equivalently, the compact group $\bE/\bN$) is Lie;
  \item or $\bN$ is finitely-generated.   \qedhere
  \end{itemize}
\end{theoremN}

\Cref{con:heisdual}, \Cref{pr:kbydist1} and \Cref{ex:discbyconn} show that one cannot, in general, drop {\it both} finite generation and the Lie condition. On the other hand, \Cref{ex:cocobad-concrete} show that even among Lie groups, this sort of type-I cocompact descent doesn't hold when the normal subgroup is not discrete. In other words, there is not much room to maneuver in dropping assumptions. 

Secondly, one can ask to what extent type-I permanence under extensions by $\bG$ {\it characterizes} compact groups. To simplify matters and fix ideas we will furthermore focus on semidirect products (rather than arbitrary extensions). The relevant concepts, then, are as follows (see \Cref{def:t1pres,def:t1linpres}).

\begin{definitionN}
  A second-countable locally compact group $\bG$ is
  \begin{itemize}
  \item {\it type-I-preserving} if $\bN\rtimes\bG$ is of type I for every type-I (second-countable) locally compact group $\bN$ acted upon continuously by $\bG$.
  \item {\it linearly type-I-preserving} if $V\rtimes \bG$ is type-I for every finite-dimensional linear representation $\bG\to GL(V)$.
  \end{itemize}
\end{definitionN}

The following is a sampling and aggregate of several results from \Cref{se:converse} (\Cref{cor:disc}, \Cref{pr:lca}, \Cref{pr:nilp}, \Cref{pr:ss} and \Cref{th:solv}).

\begin{theoremN}
  \begin{enumerate}[(1)]
  \item Discrete countable groups are
    \begin{itemize}
    \item type-I-preserving precisely if compact;
    \item linearly type-I-preserving if they have finite-index bounded-order abelian subgroups. 
    \end{itemize}
  \item Locally compact abelian groups are linearly type-I-preserving if and only if they have open, compact subgroups with bounded-order quotients. 
  \item Locally compact nilpotent groups are type-I-preserving if and only if they are compact.
  \item Connected semisimple linear Lie groups are linearly type-I-preserving.
  \item The linearly type-I-preserving connected solvable Lie groups are precisely those with compact abelianization (or equivalently, those which do not surject onto $(\bR,+)$). \qedhere
  \end{enumerate}  
\end{theoremN}

\subsection*{Acknowledgements}

Michael Brannan, Siegfried Echterhoff, Amaury Freslon, Marc Rieffel, Adam Skalski and Ami Viselter have all helped me with numerous insightful comments and pointers to the literature.

This work is partially supported through NSF grant DMS-2001128. 

\section{Preliminaries}\label{se.prel}

We work extensively with (post)liminal $C^*$-algebras, as in \cite[Definitions 4.2.1 and 4.3.1]{dixc}. Postliminality is also referred to as being {\it type I}; the subject has been studied extensively, and the relevant background is available in many good sources: \cite[Chapter 6]{ped-aut}, \cite[\S IV.1]{blk}, or the already-cited \cite{dixc} (especially Chapters 4 and 9 therein).

Of special interest will be the universal $C^*$-algebras $C^*(\bG)$ of locally compact groups $\bG$. As usual (e.g. \cite[\S 13.4]{dixc}), we refer to the groups as `type-I' when these $C^*$-algebras are such. The pertinent literature on induced representations \cite{mack-ind-1,mack-ind-2} and its applications to being type-I \cite{am} will emerge more fully in the course of the discussion below.

Given how central countability/separability assumptions are to type theory (as made clear, for instance, in \cite[\S 6.8.9 and \S 6.9]{ped-aut}), the reader is encouraged to assume all $C^*$-algebras separable and all locally compact groups second-countable (sometimes also termed `separable' in the literature, in this context \cite[Introduction, Section 3]{am}).

Hats atop locally compact groups, as in $\widehat{\bG}$, do double duty, denoting
\begin{enumerate}[(a)]
\item Pontryagin duals: \cite[\S VII.3]{tak2}, say, for the classical case of locally compact abelian groups, and \cite[\S 8]{kvcast} for the quantum version;
\item and the set of isomorphism classes of unitary irreducible representations, or the {\it spectrum} \cite[\S 3.1.5]{dixc} of the full $C^*$-algebra $C^*(\bG)$.
\end{enumerate}
The overloaded notation seems justified on several counts:
\begin{itemize}
\item Both uses are fairly well entrenched: we have already cited a few sources for Pontryagin duals, and $\widehat{\bG}$ stands in for the spectrum of $C^*(\bG)$ in \cite[18.1.1]{dixc}, \cite[p.30]{am}, \cite[\S 2.3, Definition (5)]{mack-unit}, \cite[Definition 1.46]{kt}, and doubtless many other places.
\item The two meanings converge for locally compact abelian (LCA, for short) groups: irreducible unitary representations are in that case 1-dimensional, and thus precisely {\it characters} $\bG\to \bS^1$ (i.e. elements of the Pontryagin dual).
\item Finally, at no point will the notational overlap be confusing: context will always suffice to distinguish meaning.
\end{itemize}

\section{Cocompact embeddings}\label{se:cls}

\subsection{Positive results}\label{subse:pos}

Recall (e.g. \cite[Corollaries 1 and 2]{land-sph}) that crossed products by compact groups preserve (post)liminality. In particular, whenever a compact group $\bG$ acts on a type-I group $\bN$, the semidirect product $\bN\rtimes \bG$ is again type-I. A generalization of this result will handle arbitrary extensions
\begin{equation*}
  \{1\}\to \bN\to \bE\to \bG\to \{1\}
\end{equation*}
via the {\it twisted covariant systems} of \cite{gr-tw}. We briefly review the constructions (\cite[\S 1]{gr-tw} and \cite[\S 3]{qg-full}).

\begin{definition}\label{def:twisted}
  Let $\bN\trianglelefteq \bE$ be a closed normal embedding of locally compact groups and $A$ a $C^*$-algebra.

  A {\it twisted (or $\bN$-twisted) action} attached to this data is a pair $(\alpha,\tau)$ consisting of
  \begin{itemize}
  \item a $\bG$-action $\alpha$ on $A$;
  \item a strictly-continuous morphism $\tau:\bN\to U(M(A))$ (unitary group of the multiplier algebra of $A$);
  \item so that $\tau$ intertwines the conjugation action of $\bG$ on $\bN$ and the action $\alpha$;
  \item and furthermore such that $\alpha|_{\bN}$ is conjugation by $\tau$.
  \end{itemize}

  We depict a twisted action as the entire package as $(A, \bE, \alpha, \tau)$, or again $(A, \bE, \bN, \alpha,\tau)$ in order to highlight $\bN$.

  One can then define {\it covariant representations} of this data, meaning Hilbert-space representations of both $\bE$ and $A$ appropriately compatible with $\alpha$ and $\tau$, and introduce the {\it twisted crossed product} $A\rtimes_{\alpha,\tau}\bE$ attached to such a datum (denoted by $A\widetilde{\times}_{\tau} \bE$ on \cite[p.42]{qg-full}, with $\bG$ in place of $\bE$) as the universal $C^*$-algebra carrying a $(A, \bE, \bN, \alpha,\tau)$-covariant representation.
\end{definition}

\begin{theorem}\label{th:tw}
  Let $\bN\trianglelefteq \bE$ be a closed normal subgroup of a second-countable locally compact group with compact corresponding quotient $\bG:=\bE/\bN$, and $(A,\bE,\bN,\alpha,\tau)$ an $\bN$-twisted action.
  
  If $A$ is separable and (post)liminal then so is the corresponding twisted crossed product $A\rtimes_{\alpha,\tau}\bE$.
\end{theorem}
\begin{proof}  
  The argument proceeds by duality. According to \cite[Proposition 3.1]{qg-full} the twisted crossed product
  \begin{equation*}
    B:=A\rtimes_{\alpha,\tau}\bE
  \end{equation*}
  admits a coaction by $\bG=\bE/\bN$ in the sense of \cite[Definition 2.1]{qg-full} and hence one can form the corresponding crossed product $B\rtimes \widehat{\bG}$ (the symbol is simply `$\times$' in \cite[Definition 2.3 (v)]{qg-full}). The relevant duality result, \cite[Theorem 3.6]{qg-full}, then says that we have
  \begin{equation*}
    B\rtimes \widehat{\bG} = (A\rtimes_{\alpha,\tau}\bE)\rtimes \widehat{\bG}\cong A\otimes \cK(L^2(\bG)),
  \end{equation*}
  and the conclusion follows: $\bG$ being compact, we have an embedding $B\subset B\rtimes \widehat{\bG}$, and hence an embedding of $B$ into the (post)liminal $C^*$-algebra $A\otimes \bK(L^2(\bG))$. Since (post)liminality is inherited by $C^*$-subalgebras (\cite[Propositions 4.2.4 and 4.3.5]{dixc}), we are done.
\end{proof}

In particular, this recovers postliminality lifting along normal cocompact embeddings of locally compact groups, by different means than via \cite[Theorem 1]{kal1}.

\begin{corollary}\label{cor:ext}
  If
  \begin{equation}\label{eq:ext}
    \{1\}\to \bN\to \bE\to \bG\to \{1\}
  \end{equation}
  is an extension of second-countable locally compact groups with $\bN$ of type I and $\bG$ compact, $\bE$ too must be type-I. 
\end{corollary}
\begin{proof}
  Apply \Cref{th:tw} to the $\bN$-twisted action $(C^*(\bN),\bE,\bN,\alpha,\tau)$ attached to the extension \Cref{eq:ext} as in \cite[p.199, Corollary]{gr-tw}, whose underlying crossed product $C^*(\bN)\rtimes_{\alpha,\tau} \bE$ is nothing but the group algebra $C^*(\bE)$.
\end{proof}

While it might seem reasonable to lift the type-I property along possibly-non-normal cocompact (closed) embeddings \cite[Conjecture 1]{kal2}, it certainly will not {\it descend} from a locally compact group $\bE$ to a cocompact closed subgroup $\bN\le \bE$:

\begin{example}\label{ex:latt}
  Recall that a connected semisimple Lie group $\bE$ always has a cocompact discrete closed subgroup $\bN$: see for instance \cite[Theorem C]{bor-ck} or \cite[Theorem 14.1]{ragh} (which sources refer to cocompact subgroups as {\it uniform}).

  Now, being connected semisimple Lie, $\bE$ is type-I (\cite[Theorem 7 and Introduction, p.186]{hc}). Its {\it countable} subgroup $\bN$ though, by Thoma's theorem (\cite[Satz 6]{thoma} or \cite[Theorem 1]{tt}), can only be type-I if it has a finite-index abelian subgroup. The Zariski-density of $\bN\le \bE$ \cite[Theorem 5.5]{ragh} when $\bE$ has no compact quotients would then entail the abelianness of $\bE$, contradicting semisimplicity.
\end{example}

\begin{remark}\label{re:cofin}
  On several occasions we use the fact that the type-I property survives under passage to open subgroups: this is \cite[Proposition 2.4]{kal2}, and follows alternatively by noting that
  \begin{itemize}
  \item $C^*$-subalgebras of type-I $C^*$-algebras are again type-I \cite[Proposition 4.3.5]{dixc};
  \item and an open embedding of groups induces an embedding of full $C^*$-algebras \cite[Lemma 1.1]{rief}.
  \end{itemize}
  This latter result would have been unavailable in the slightly earlier \cite{kal2}, hence the slightly more complicated proof given there.

  In particular, since cofinite subgroups are both cocompact and open, being type-I lifts and descends along cofinite embeddings.
\end{remark}

Before stating a partial converse to \Cref{cor:ext}, we need an auxiliary observation and some language.

\begin{definition}\label{def:centemb}
  For locally compact groups $\bG_i$, $i=1,2$ and $\bD$ a {\it central pushout} $\bG_1\coprod_{\bD}\bG_2$ is a quotient of $\bG_1\times \bG_2$ obtained by identifying
  \begin{equation*}
    \iota_i(d)\in \bG_i\subset \bG_1\times \bG_2,\ i=1,2
  \end{equation*}
  for closed central embeddings $\iota_i:\bD\to \bG_i$. In particular $\bD$ is implicitly assumed abelian.   
\end{definition}

\begin{proposition}\label{pr:centemb}
  Consider a central pushout $\bE:=\bG_1\coprod_{\bD}\bG_2$ as in \Cref{def:centemb}. If $\bE$ and $\bG_2$ are type-I so is $\bG_1$.
\end{proposition}
\begin{proof}
  We are assuming $\bG_2$ is of type I. It is also clearly normal in $\bE$, so we can apply the usual Mackey machinery as outlined in \cite[Chapter I, Proposition 10.4]{am}. Write
  \begin{equation*}
    \bH:=\bG_1/\bD\cong \bE/\bG_2. 
  \end{equation*}
  The extension
  \begin{equation*}
    \{1\}\to \bD\to \bG_1\to \bH\to \{1\}
  \end{equation*}
  corresponds to a cohomology class in $H^2(\bH,\bD)$, where the cohomology groups are those introduced on \cite[p.43]{moore-ext} (see also \cite[top of p.44, before \S 2]{moore-ext} for a sketch of the extension-cohomology correspondence and \cite[Theorem 7.8]{vrd} for a detailed treatment of {\it central} extensions).

  A character
  \begin{equation*}
    \widehat{\bD}\ni \chi:\bD\to \bS^1
  \end{equation*}
  pushes this class forward to
  \begin{equation*}
    c_{\chi}\in H^2(\bH,\bS^1),
  \end{equation*}
  so that we can now speak of {\it projective $\bH$-$c_{\chi}$-representations} (\cite[p.22]{am} or \cite[\S VI.2]{vrd}).

  The centrality assumption implies that the action of $\bH$ on the spectrum $\widehat{\bD}$ is trivial, and \cite[Chapter I, Proposition 10.4]{am} then equates the type-I property for $\bG_1$ with the requirement that all projective $\bH$-$c_{\chi}$-representations be type-I, for arbitrary $\chi\in \widehat{\bG}$.

  Next, observe, that every $\chi\in \widehat{\bG}$ (or rather a sum of copies thereof) arises as the restriction to $\bD$ of a factor $\bG_2$-representation $\rho$: consider the {\it induced representation} $\mathrm{Ind}_{\bD}^{\bG_2}\chi$ (\cite[Chapter I, Section 9]{am} or \cite[\S 2.3]{kt}), and then take for $\rho$ any of the constituents in a {\it direct-integral decomposition}
  \begin{equation*}
    \mathrm{Ind}_{\bD}^{\bG_2}\chi\cong \int_{\widehat{\bG_2}}\rho_x\ \mathrm{d}\mu(x);
  \end{equation*}
  see \cite[Theorem 8.4.2]{dixc} or \cite[Theorem 4.12.4]{ped-aut}. Fix $\chi\in \widehat{\bD}$ and $\rho\in \widehat{\bG_2}$ restricting to the former.

  Now apply the Mackey machine to the normal inclusion $\bG_2\trianglelefteq \bE$ instead; here too, $\bH\cong \bE/\bG_2$ acts trivially on $\widehat{\bG_2}$. By direct examination of its construction in \cite[Chapter I, Proposition 10.3]{am}, it is clear that the {\it Mackey obstruction cocycle} in
  \begin{equation*}
    H^2(\bE/\bG_2,\bS^1)\cong H^2(\bH,\bS^1)
  \end{equation*}
  attached to $\rho$ is nothing but $c_{\chi}$. The type-I assumption on $\bE$ together with \cite[Chapter I, Proposition 10.4]{am} then tell us that indeed projective $c_{\chi}$-representations of
  \begin{equation*}
    \bH\cong \bE/\bG_2\cong \bG_1/\bD
  \end{equation*}
  are type-I, so we are done. 
\end{proof}

As a consequence, we obtain the following partial converse to \Cref{cor:ext}. It also shows that \Cref{ex:latt} could not have gone through (for {\it Lie} groups; an important constraint) if $\bN$ were normal.

\begin{proposition}\label{pr:cpctliequot}
  If \Cref{eq:ext} is a type-I extension with $\bN$ discrete and $\bG$ compact Lie then $\bN$ too must be type-I.
\end{proposition}
\begin{proof}
  First, because $\bG$ is compact Lie, it has finitely many connected components. By \Cref{re:cofin}, there is no harm in assuming it connected upon passing to a finite-index subgroup.
  
  Being an extension of Lie groups, $\bE$ itself is Lie \cite[Theorem 3.1]{gls} and hence its connected component $\bE_0$ is open. Since the map $\bE\to \bE/\bN\cong \bG$ is open, it maps $\bE_0$ onto an open subgroup. Because $\bG$ is connected, $\bE_0$ {\it surjects} onto $\bG$.

  This means that $\bE$ equals $\bN \bE_0$. The intersection $\bD:=\bN\cap \bE_0$ is discrete and normal in the connected group $\bE_0$, so it must be central there. It is also central in $\bN$ (and hence globally, in $\bE$) because the connected group $\bE_0$ normalizes and hence centralizes the discrete group $\bN$.

  We can now recover $\bE$ as a central pushout
  \begin{equation*}
    \bE\cong \bN\coprod_{\bD}\bE_0.
  \end{equation*}
  It is of type I by assumption, and $\bE_0$ is of type I by \Cref{cor:ext} because it contains an abelian, normal, cocompact subgroup $\bD$. The conclusion follows from \Cref{pr:centemb}.
\end{proof}

Given an extension \Cref{eq:ext} with $\bN$ discrete and $\bG$ compact connected but {\it not} Lie, the neat decomposition phenomena leveraged in the proof of \Cref{pr:cpctliequot} might not obtain. Specifically, the discrete normal subgroup $\bN\trianglelefteq \bE$ and the connected component $\bE_0$ might not generate a closed subgroup of $\bE$.

\begin{example}\label{ex:qhat}
  Take the group $\bE$ in the extension \Cref{eq:ext} to be the product $\bR\times \overline{\bZ}$, where the second factor is the {\it profinite completion}
  \begin{equation*}
    \overline{\bZ}\cong \widehat{\bQ/\bZ}
  \end{equation*}
  of $\bZ$ (for the latter notion see \cite[Example 2.1.6 (2)]{rz}, with the caveat that that reference uses a hat for the bar; this would conflict with the present paper's use of hats to denote {\it duals}). 

  The normal subgroup $\bN\le \bE$ of \Cref{eq:ext} will be the ``diagonal'' copy of $\bZ$:
  \begin{equation*}
    \bN:=\{(n,n)\in \bR\times \overline{\bZ}\ |\ n\in \bZ\}. 
  \end{equation*}
  The quotient $\bG=\bE/\bN$ of \Cref{eq:ext} is a central pushout $\bR\coprod_{\bZ}\overline{\bZ}$ in the sense of \Cref{def:centemb}, and thus dual to the fibered product $\bR\times_{\bS_1}(\bQ/\bZ)$ of the two familiar maps
  \begin{equation*}
    \bR\to \bS^1\quad\text{and}\quad \bQ/\bZ\to \bS^1
  \end{equation*}
  (the latter identifying $\bQ/\bZ$ with the torsion subgroup of the circle, i.e. the group of roots of unity). This is easily seen to be precisely $\bQ$, and hence
  \begin{equation*}
    \bG=\bE/\bZ\cong \widehat{\bQ}.
  \end{equation*}
  We thus have an (abelian) extension
  \begin{equation*}
    \{1\}\to \bZ\to \bE\to \widehat{\bQ}\to \{1\}
  \end{equation*}
  with $\bZ$ discrete and $\widehat{\bQ}$ compact and connected (being dual to a torsion-free discrete group \cite[Corollary 8.5]{hm3}). Note, though, that the specific copy of $\bZ\subset \bE$ we chose maps densely into
  \begin{equation*}
    \bE/\bE_0 = \bE/\bR\cong \overline{\bZ},
  \end{equation*}
  so that the subgroup $\bN\bE_0\le \bE$ featuring in the proof of \Cref{pr:cpctliequot}, this time around, is proper and dense.
\end{example}

\subsection{Cocompact counterexamples}\label{subse:coco-counter}

We gather a number of examples of the type-I property failing to descend along cocompact normal embeddings. The first batch, based on \Cref{con:cocobad}, will produce extensions with {\it connected} compact quotient.

\begin{construction}\label{con:cocobad}
  Let $\bA$ be a torsion-free discrete abelian group fitting into a non-type-I central extension
  \begin{equation}\label{eq:s1ea}
    \{1\}\to \bS^1\xrightarrow[]{\quad\phantom{\pi}\quad} \bE\xrightarrow[]{\quad\pi\quad} \bA\to \{1\}. 
  \end{equation}
  Denote by $\bK:=\widehat{\bA}$ its compact Pontryagin dual, and form the semidirect product $\bE\rtimes \bK$ with respect to the left action
  \begin{equation*}
    \bK\times \bE\xrightarrow[]{\quad \triangleright\quad} \bE
  \end{equation*}
  defined by  
  \begin{equation*}
    \chi\triangleright x:=\chi(\pi(x))x,\ \forall x\in \bE,\ \chi\in \bK,
  \end{equation*}
  where $\chi(\pi(x))\in \bS^1$ is regarded as an element of the central circle in \Cref{eq:s1ea}. 

  On the one hand, we are assuming that $\bE$ itself is not type-1. On the other, $\bE\rtimes \bK$ will be. To see this, note that it can also be recovered as an extension
  \begin{equation}\label{eq:s1kea}
    \{1\}\to \bS^1\times \bK\to  \bE\rtimes \bK\to \bA\to \{1\},
  \end{equation}
  to which we can apply the usual Mackey analysis (as recalled, say, in \cite[\S 4.3]{kt} or \cite[Chapter I]{am}). \Cref{eq:s1kea} induces an action of $\bA$ on the discrete character space
  \begin{equation*}
    \widehat{\bS^1\times \bK}\cong \bZ\times \bA
  \end{equation*}
  that is easily seen to be
  \begin{itemize}
  \item trivial on
    \begin{equation*}
      \bA\cong \{0\}\times \bA\subset \bZ\times \bA,
    \end{equation*}
    i.e. on those characters that vanish on $\bS^1$, and hence lift to characters of the abelian quotient $\bA\times \bK$ of $\bE\rtimes \bK$;
  \item and free elsewhere: on the component
    \begin{equation*}
      \bA\cong \{n\}\times \bA\subset \bZ\times \bA\cong \widehat{\bS^1\times \bK},\ n\ne 0
    \end{equation*}
    the action of $a\in \bA$ is translation by $na\in \bA$. Since we are assuming torsion-freeness, the action is indeed free.
  \end{itemize}
  All spaces in sight (the character space $\bZ\times \bA$ being acted upon and the acting group $\bA$) are discrete, so the regular embedding condition of \cite[p.186, Definition]{mack-unit} holds. One can then apply \cite[Chapter I, Proposition 10.4]{am} to conclude that $\bE\rtimes \bK$ is indeed of type I.
\end{construction}

As for producing central extensions \Cref{eq:s1ea} needed in \Cref{con:cocobad}, we first recall the following familiar construction (see e.g. the {\it groups of Heisenberg type} of \cite[discussion preceding Remarks 4]{milnes-heis}).

\begin{definition}\label{def:gendbl}
  Let
  \begin{equation*}
    p:\bA_l\times \bA_r\to \bB
  \end{equation*}
  be a continuous bilinear map of locally compact abelian groups. The {\it Heisenberg double} $D(p)$ associated to it is the semidirect product
  \begin{equation*}
    (\bB\times \bA_l)\rtimes \bA_r
  \end{equation*}
  induced by the action
  \begin{equation*}
    \bA_r\times (\bB\times \bA_l)\ni (a_r,b,a_l)\mapsto (p(a_l,a_r)b, a_l)\in \bB\times \bA_l.
  \end{equation*}
  It can alternatively be described as a semidirect product $(B\times \bA_r)\rtimes \bA_l$ (by interchanging the roles of the two $\bA_{\bullet}$s) or as a central extension
  \begin{equation*}
    \{1\}\to \bB\to D(p)\to \bA_l\times \bA_r\to \{1\}.
  \end{equation*}  
\end{definition}

We can now elaborate on the extensions \Cref{eq:s1ea}. 

\begin{construction}\label{con:nt1}
  The discrete abelian group $\bA$ of \Cref{con:cocobad} will be of the form $\bA_l\times \bA_r$ and $\bE$ will be a Heisenberg double $D(p)$ (per \Cref{def:gendbl}) for a pairing
  \begin{equation}\label{eq:pairalr}
    p:\bA_l\times \bA_r\to \bS^1
  \end{equation}
  whose associated morphism
  \begin{equation}\label{eq:phia1a2}
    \varphi:\bA_r\to \widehat{\bA_l} = \mathrm{Hom}\left(\bA_l,\bS^1\right)
  \end{equation}
  has infinite (or equivalently, non-discrete) image.

  The action of $\bA_r$ on the dual
  \begin{equation*}
    \widehat{\bS^1\times \bA_l}\cong \bZ\times \widehat{\bA_l}
  \end{equation*}
  is such that $a_r\in \bA_r$ operates on the copy of $\widehat{\bA_l}$ indexed by $n\in \bZ$ by translation with $\varphi(na_r)\in \widehat{\bA_l}$.

  The hypotheses of \cite[Theorem 1]{glm} are met and condition (6) there is clearly violated: being infinite, at least some of the orbits of this action are not homeomorphic to (discrete) quotients of $\bA_2$. But then we can conclude that $D(p)$ is not of type I by \cite[Chapter II, Corollary to Proposition 9]{am}: the discreteness condition therein holds, hence the conclusion.
\end{construction}

There is no shortage of {\it concrete} incarnations of \Cref{con:cocobad} (jointly with \Cref{con:nt1}):

\begin{examples}\label{ex:cocobad-concrete}
  One can take the morphism \Cref{eq:phia1a2} to be 
  \begin{itemize}
  \item the embedding
    \begin{equation*}
      \bZ\to \bS^1\cong \widehat{\bZ}
    \end{equation*}
    sending the generator to a non-root of unity; this was (essentially) the example produced by an anonymous referee in discussing an earlier version of this work.
  \item or more generally, any morphism from $\bZ$ to the discrete dual of a torsion-free abelian group, sending a generator to a non-torsion element.

    {\it Any} (non-trivial) torsion-free abelian group will do, since the dual of such a group is compact connected \cite[Corollary 8.5]{hm3} and hence cannot be torsion (e.g. \cite[\S 5, Exercise 2]{mor-pont}).
  \item or any non-trivial morphism $\bQ\to \widehat{\bQ}$. 

    Such morphisms exist in abundance, since $\widehat{\bQ}$ is {\it divisible} and hence any morphism $\bZ\to \widehat{\bQ}$ extends (\cite[Corollary 8.5 and discussion preceding Lemma 1.21]{hm3}). And any such morphism will meet the requirements (i.e. have infinite image), since $\widehat{\bQ}$ is torsion-free \cite[Corollary 8.5]{hm3}.
  \end{itemize}
  In summary, all examples produced this way will give non-type-I $\bE$ with type-I semidirect products $\bE\rtimes \bK$ for compact connected $\bK$.
\end{examples}

A different class of examples will produce type-I semidirect products $\bD\rtimes \bK$ with
\begin{itemize}
\item $\bD$ discrete, nilpotent but {\it not} of type I;
\item and $\bK$ profinite. 
\end{itemize}

First, a follow-up on \Cref{def:gendbl}:

\begin{definition}\label{def:dbla}
  Let $\bA$ be a locally compact abelian group.
  \begin{itemize}
  \item Its {\it Heisenberg double} $D(\bA)$ is the $D(p)$ of \Cref{def:gendbl} for the canonical pairing
    \begin{equation}\label{eq:canps1}
      p:\bA\times \widehat{\bA}\to \bS^1.
    \end{equation}
  \item The {\it restricted (Heisenberg) double} $D_r(\bA)$ is $D(p_r)$, where
    \begin{equation*}
      p_r:\bA\times\widehat{\bA}\to \bP
    \end{equation*}
    is the corestriction of \Cref{eq:canps1} to the closed subgroup of $\bS^1$ its image generates.
  \end{itemize}
  The distinction between $D$ and $D_r$ of course only makes a difference for {\it finite} $\bA$: $D_r(\bA)$ is then finite, while $D(\bA)$ is not.  
\end{definition}


\begin{remark}\label{re:heisplit}
  For finite cyclic $\bA\cong \bZ/n$ (in which case $\widehat{\bA}\cong \bZ/n$ as well) the restricted double $D_r(\bA)$ is nothing but the finite Heisenberg group of order $n^3$ denoted by $\cH(\bZ_n)$ in \cite[\S 2]{schul-heis}.

  An arbitrary finite abelian $\bA$ decomposes as a direct sum of finite cyclic groups. Any decomposition $\bA\cong \bA_1\oplus \bA_2$ gives provides its dual decomposition for $\widehat{\bA}$ compatible with pairings. This then gives embeddings
  \begin{equation*}
    D_r(\bA_i)\subseteq D_r(\bA),\ i=1,2,
  \end{equation*}
  which induce a surjection
  \begin{equation*}
    D_r(\bA_1)\oplus D_r(\bA_2)\to D_r(\bA).
  \end{equation*}
  All in all, we have a surjection of the form
  \begin{equation*}
    \bigoplus_i D_r(\bZ/n_i)\to D_r(\bA)
  \end{equation*}
  realized by identifying the centers of the summands of the domain with cyclic subgroups of the same circle; in particular, that surjection restricts to an embedding for each individual $D_r(\bZ/n)$. This will all be implicit in much of the discussion below.
\end{remark}

\begin{construction}\label{con:heisdual}
  Let $(\bA_n)_{n\in \bZ_{>0}}$ be any infinite family of non-trivial finite abelian groups and set
  \begin{equation*}
    \bD_n:=D_r(\bA_n),\quad \bD:=\bigoplus_n \bD_n.
  \end{equation*}
  Each $\bA_n$ acts on its own $D_r(\bA_n)$ by conjugation, giving an action by the compact {\it product}
  \begin{equation*}
    \bK:=\prod_n \bA_n
  \end{equation*}
  on $\bD$.

  The group we are interested in will be the {\it family-wise equivariant restricted double}
  \begin{equation*}
    \bE_r\left((\bA_n)_n\right):=\bD\rtimes \bK = \left(\bigoplus_n \bD_n\right)\rtimes\left(\prod_n \bA_n\right). 
  \end{equation*}
  That $\bD$ is not virtually abelian and hence not of type I is easily seen, being an infinite sum of non-abelian groups. Alternatively, instead of appealing to Thoma's \cite[Theorem 1]{tt}, note that every $C^*(\bD_n)$ has some matrix-algebra quotient $M_{d_n}$ for $d_n\ge 1$ because $\bD_n=D_r(\bA_n)$ is not abelian, and hence we have a surjection
  \begin{equation*}
    C^*(\bD)\cong \bigotimes_n C^*(\bD_n)\twoheadrightarrow \bigotimes_n M_{k_n}
  \end{equation*}
  between the respective {\it infinite tensor products} \cite[Definition XIV.1.5]{tak3}. The right-hand side is one of the {\it Glimm algebras} \cite[\S 6.5]{ped-aut}, well known not to be of type I \cite[Theorem 6.5.15]{ped-aut}.

  It thus remains to argue that $\bE:=\bE_r\left((\bA_n)_n\right)$ is type-I. 
\end{construction}

\begin{proposition}\label{pr:kbydist1}
  For any family $(\bA_n)_n$ of non-trivial finite abelian groups the family-wise equivariant restricted double $\bE:=\bE_r\left((\bA_n)_n\right)$ of \Cref{con:heisdual} is of type I.
\end{proposition}
\begin{proof}
  Decompose the finite-dimensional group algebra of each individual $\bD_n=D_r(\bA_n)$ as a product of matrix algebras, as in \Cref{con:heisdual}. That decomposition is of course invariant under the inner action by $\bA_n$, so every minimal central projection in
  \begin{equation*}
    C^*(\bD_n)\subset C^*(\bD)
  \end{equation*}
  will act centrally in any representation of $\bE=\bD\rtimes \bK$. It follows that {\it factor} representations factor through surjections of the form
  \begin{align*}
    C^*(\bE) &= C^*(\bD\rtimes \bK)\\
             &\cong C^*(\bD)\rtimes \bK \\
             &\cong \left(\bigotimes_n C^*(\bD_n)\right)\rtimes \bK\numberthis \label{eq:bigtens} \\
             &\twoheadrightarrow \left(\bigotimes_n M_{k_n}\right)\rtimes \bK,
  \end{align*}
  so it is enough to focus on these quotients of $C^*(\bE)$.

  Consider an individual $\bA_n$. The quotient $C^*(\bD_n)\twoheadrightarrow M_{k_n}$ can be identified, for some quotient $\bA_n\twoheadrightarrow \bB_n$ (of order $k_n$)
  \begin{itemize}
  \item with the crossed product $C^*\left(\bB_n\right)\rtimes \widehat{\bB_n}$;
  \item so that the inner action by $\bA_n$ factors through the {\it dual action} \cite[Definition X.2.4]{tak2} by $\bB_n$ on that crossed product.
  \end{itemize}
  This is easy to see for cyclic $\bA_n$ from the explicit construction of its finite-dimensional representations \cite[\S 2]{gh-heis}, and follows in general from a decomposition of $\bA_n$ as a sum of cyclic summands, as in \Cref{re:heisplit}.

  Piecing this all together then, \Cref{eq:bigtens} is isomorphic to
  \begin{equation}\label{eq:doublecross}
    \left(\bigotimes_n C^*(\bB_n)\right)\rtimes \left(\bigoplus_n \widehat{\bB_n}\right)\rtimes \left(\prod_n \bA_n\right)
  \end{equation}
  with the rightmost action factoring through the dual action of the smaller quotient
  \begin{equation}\label{eq:atobquot}
    \bK=\prod_n \bA_n\to \prod_n \bB_n. 
  \end{equation}
  If we had {\it equality} $\bA_n=\bB_n$ (for all $n$) the resulting algebra, by {\it Takai duality} (\cite[Theorem 3.4]{takai-dual}, \cite[Theorem 7]{raeb-dual}, etc.), would be Morita equivalent to the leftmost (commutative) factor $C^*\left(\bigoplus \bB_n\right)$.
  
  In general, \Cref{eq:atobquot} has a kernel
  \begin{equation*}
    \bK_0:=\prod\bA_{n,0}
  \end{equation*}
  (say). That kernel will act centrally, and hence by a character in any factor representation. Every character $\chi_0\in\widehat{\bK_{0}}$ lifts to a character $\chi$ on $\bK$; the lift is not unique, but we choose one once and for all for each $\chi_0$. Then, in every factor representation with $\bK_{0}$ acting via $\chi_0$, we can twist the action of the larger group $\bK$ by $\chi^{-1}$ to obtain a (still factor) representation with $\bK$ acting trivially.

  All in all, we have decomposed the {\it factor spectrum} or {\it quasi-spectrum} (\cite[\S 7.2]{dixc} or \cite[\S 4.8.2]{ped-aut}) of \Cref{eq:doublecross} as a disjoint union of copies of 
  \begin{equation*}
    \left(\bigotimes_n C^*(\bB_n)\right)\rtimes \left(\bigoplus_n \widehat{\bB_n}\right)\rtimes \left(\prod_n \bB_n\right).
  \end{equation*}
  Once more, Takai duality delivers the conclusion that the $C^*$-algebra is of type I. 
\end{proof}

Per \Cref{con:heisdual} and \Cref{pr:kbydist1} we now have examples
\begin{equation*}
  \bE = \bE_r\left((\bA_n)_n\right) =
  \left(\bigoplus_n \bD_r(\bA_n)\right)
  \rtimes
  \left(\prod_n \bA_n\right)
\end{equation*}
of type-I semidirect products of non-type-I discrete groups by arbitrary products of finite abelian groups.

We can now combine this discussion with \Cref{ex:qhat} to produce counterexamples to (the conclusion of) \Cref{pr:cpctliequot} when the compact quotient $\bG$ is connected but not Lie.

\begin{example}\label{ex:discbyconn}
  Consider the group $\bE=\bR\times \overline{\bZ}$ of \Cref{ex:qhat}. Its $\overline{\bZ}$ factor decomposes the product
  \begin{equation*}
    \overline{\bZ}\cong \prod_{\text{primes } p}\bZ_p
  \end{equation*}
  of the groups
  \begin{equation*}
    \bZ_p:=\varprojlim_n \bZ/p^n
  \end{equation*}
  of {\it $p$-adic integers} \cite[Examples 2.1.6 (2) and 2.3.11]{rz}. Its quotient
  \begin{equation}\label{eq:ozonzps}
    \overline{\bZ}\cong \prod_p\bZ_p\twoheadrightarrow \prod_p \bZ/p
  \end{equation}
  acts, as in the preceding discussion, on some discrete non-type-I group $\bD_1$ so that $\bD_1\rtimes\left(\prod_p \bZ/p\right)$ is type-I. If now that action is extended (i.e. pulled back) to all of $\overline{\bZ}$ along \Cref{eq:ozonzps} and then also to $\bE$ along the surjection $\bE\to \overline{\bZ}$, we have an $\bE$-action on $\bD_1$. As in the proof of \Cref{pr:kbydist1}, the semidirect product
  \begin{equation*}
    \bD_1\rtimes \bE\cong \bR\times (\bD_1\rtimes\overline{\bZ})
  \end{equation*}
  is still of type I. 

  Now recall the diagonally-embedded $\bZ\subset \bE$ in \Cref{ex:qhat}. {\it Its} action on $\bD_1$ will provide us with a normal subgroup
  \begin{equation*}
    \bD:=\bD_1\rtimes \bZ \trianglelefteq \bD_1\rtimes \bE,
  \end{equation*}
  with quotient
  \begin{equation*}
    (\bD_1\rtimes \bE)/(\bD_1\rtimes \bZ)\cong \bE/\bZ\cong \widehat{\bQ}. 
  \end{equation*}
  This is compact and connected, so \Cref{pr:cpctliequot} fails even for {\it connected} compact quotients $\bG$ if they are not Lie.
\end{example}

For completeness, we end this section with a variant of \Cref{pr:cpctliequot} that leaves the compact quotient unrestricted but instead places the requirements on the discrete normal subgroup $\bN\trianglelefteq \bE$.

\begin{proposition}\label{pr:fgok}
  If \Cref{eq:ext} is a type-I extension with $\bN$ discrete finitely-generated and $\bG$ compact then $\bN$ too must be type-I.
\end{proposition}
\begin{proof}
  Recall (e.g. \cite[discussion preceding Theorem 10.89]{hm3}) that a locally compact group is {\it almost connected} if its quotient by its identity connected component is compact. 
  
  Being locally compact, $\bE$ has an open almost-connected subgroup $\bE_1$ \cite[Lemma 2.3.1]{mz}. That group will then map onto an open (hence cofinite) subgroup of the compact group $\bG=\bE/\bN$, so by passing to a cofinite subgroup (harmlessly, by \Cref{re:cofin}) we can assume that
  \begin{equation*}
    \bE_1\text{ surjects onto }\bE/
    \bN\Rightarrow
    \bE = \bN\bE_1. 
  \end{equation*}
  Being almost-connected, $\bE_1$ is {\it approximable by Lie groups} (or {\it pro-Lie} \cite[Chapter 3]{hm-pl-bk}): every neighborhood of the identity contains a closed (hence compact) normal subgroup so that the resulting quotient is Lie \cite[\S 4.6, Theorem]{mz}. It follows from \cite[Corollary 8.3]{hm-struct-plie} that $\bE_1$ is of the form
  \begin{equation}\label{eq:e1decomp}
    \bE_1 = \bE_0 \bK,\quad \bE_0\text{ the connected component of }\bE,\ \bK\text{ profinite}.
  \end{equation}
  The inner action of the connected group $\bE_0$ on the discrete group $\bN$ is of course trivial, while that of the {\it compact} group $\bK$ has finite orbits. In particular, the centralizer in $\bK$ of the finitely many generators of $\bN$ has finite index, so upon passing to a cofinite subgroup again we can assume that all of \Cref{eq:e1decomp} centralizes $\bN$.

  Note then that $\bE$ is a type-I central pushout
  \begin{equation*}
    \bE = \bN\coprod_{\bN\cap \bE_1}\bE_1,
  \end{equation*}
  and \Cref{pr:centemb} will allow us to conclude as soon as we prove $\bE_1$ is of type I. It contains $\bN\cap \bE_1$ as a cocompact discrete normal subgroup, so by \Cref{cor:ext} it is enough to prove that $\bN\cap \bE_1$ is of type I, i.e. virtually abelian. That, in turn, follows from \Cref{le:almostco}.
\end{proof}

\begin{lemma}\label{le:almostco}
  Let $\bN\trianglelefteq \bE$ be a closed, normal, discrete subgroup of an almost connected locally compact group. There is a cofinite subgroup $\bE_1\le \bE$ such that
  \begin{equation*}
    \bN\cap \bE_1\le \bE_1
  \end{equation*}
  is central.
\end{lemma}
\begin{proof}
  Using once more the fact that $\bE$ is pro-Lie\cite[\S 4.6, Theorem]{mz}, we can select a compact normal subgroup $\bK\trianglelefteq \bE$, sufficiently small that it intersects $\bN$ trivially, with Lie quotient $\bE/\bK$.

  That quotient has finitely many connected components by the almost-connectedness of $\bE$; there is thus a group
  \begin{equation*}
    \bK\le \bE_1\le \bE,
  \end{equation*}
  cofinite in $\bE$, such that $\bE_1/\bK$ is Lie and connected. Because $\bK$ is compact the image $\overline{\bN_1}$ of
  \begin{equation*}
    \bN_1:=\bN\cap \bE_1
  \end{equation*}
  in $\bE_1/\bK$ is still closed, discrete, and normal. But because $\bE_1/\bK$ is connected, $\overline{\bN_1}$ must be central therein. This means that $\bN_1$ is central in $\bE_1$ modulo $\bK$, hence central period by the assumption that $\bN_1$ and $\bK$ intersect trivially. 
\end{proof}

\section{Characterizing groups via type-I permanence}\label{se:converse}

We focus here on results converse to those of \Cref{se:cls}: the goal is to deduce various group-theoretic or topological properties (e.g. compactness) {\it assuming} extensions by the group in question preserve the type-I property. It will be convenient to have a short phrase indicating this quality.

\begin{definition}\label{def:t1pres}
  A second-countable locally compact group $\bG$ is {\it type-I-preserving} if for any action of $\bG$ on a type-I second-countable locally compact group $\bN$ the semidirect product $\bN\rtimes \bG$ is again type-I.
\end{definition}

Since one salient procedure that will produce a semidirect product as in \Cref{def:t1pres} is to take for $\bN$ the underlying space of a finite-dimensional $\bG$-representation, we single out a more restrictive notion. 

\begin{definition}\label{def:t1linpres}
  A second-countable locally compact group $\bG$ is {\it linearly type-I-preserving} if for any finite-dimensional linear representation of $\bG$ on $\bR^n$ the semidirect product $\bR^n\rtimes \bG$ is type-I.
\end{definition}

\begin{remark}\label{re:ist1}
  Note that (linearly) type-I-preserving groups are always type-I, since we can apply the defining property to the trivial action on the trivial group.
\end{remark}

\begin{lemma}\label{le:quot}
  If a second-countable locally compact group $\bG$ is (linearly) type-I-preserving, so, respectively, is any quotient $\bG/\bH$ by a closed normal subgroup $\bH\trianglelefteq \bG$.
\end{lemma}
\begin{proof}
  Consider an action of $\bG/\bH$ on a second-countable locally compact type-I group $\bN$. It restricts to an action of $\bG$ along the quotient $\bG\to \bG/\bH$, and we correspondingly have a quotient
  \begin{equation*}
    \bN\rtimes \bG\to \bN\rtimes (\bG/\bH)
  \end{equation*}
  of semidirect products and hence one between their attached full $C^*$-algebras. Since $C^*(\bN\rtimes \bG)$ is assumed type-I, its quotient $C^*(\bN\rtimes (\bG/\bH))$ must be too \cite[Proposition 4.3.5]{dixc}.

  The proof applies uniformly whether $\bN$ is the underlying vector space of a finite-dimensional $\bG/\bH$-representation or not, so this argument delivers both claims.
\end{proof}

We also have the following simple remark on how the various type-I-preservation conditions relate to one another.

\begin{lemma}\label{le:t1impl}
  Consider the following conditions for a second-countable locally compact group $\bG$.
  \begin{enumerate}[(a)]
  \item\label{item:3} $\bG$ is compact.
  \item\label{item:4} $\bG$ is type-I-preserving.
  \item\label{item:5} $\bG$ is linearly type-I-preserving. 
  \end{enumerate}
  We then have
  \begin{equation*}
    \text{\Cref{item:3} $\Rightarrow$ \Cref{item:4} $\Rightarrow$ \Cref{item:5}}.
  \end{equation*}
\end{lemma}
\begin{proof}
  The first implication follows from \Cref{cor:ext}, while the second is immediate.
\end{proof}

A few other fairly straightforward observations will help transition between groups.

\begin{lemma}\label{le:changegp}
  Let $\bH\le \bG$ be a closed cocompact embedding of second-countable locally compact groups.
  \begin{enumerate}[(1)]
  \item\label{item:8} If $\bH$ is (linearly) type-I-preserving then so, respectively, is $\bG$.
  \item\label{item:9} If $\bH\le \bG$ is furthermore {\it cofinite} and $\bG$ is linearly type-I-preserving then so is $\bH$.
  \end{enumerate}
\end{lemma}
\begin{proof}
  An action of $\bG$ on a locally compact group $\bM$ (resulting from a linear representation or not) restricts to $\bH$ and we have a cocompact embedding
  \begin{equation*}
    \bM\rtimes \bH\le \bM\rtimes \bG. 
  \end{equation*}
  Claim \Cref{item:8} now follows from \Cref{cor:ext}.
  
  As for claim \Cref{item:9}, suppose $\bG$ is linearly type-I-preserving and let $\pi:\bH\to GL(V)$ be a finite-dimensional representation and
  \begin{equation*}
    \rho:=\mathrm{Ind}_{\bH}^{\bG}\pi:\bG\to GL(W)
  \end{equation*}
  its induction to $\bG$ (unproblematic to define, given that the embedding $\bH\le \bG$ is cofinite). Then
  \begin{itemize}
  \item $W\rtimes \bG$ is type-I by assumption;
  \item hence its cofinite closed subgroup $W\rtimes \bH$ is type-I by \Cref{re:cofin};
  \item so its quotient
    \begin{equation*}
      W\rtimes \bH\to V\rtimes \bH
    \end{equation*}
    is type-I because that property survives under taking group quotients. 
  \end{itemize}
  This concludes the proof.
\end{proof}

And also:

\begin{lemma}\label{le:prod}
  Let $\bG$ and $\bH$ be two locally compact groups, with $\bH$ type-I-preserving. If $\bG$ is (linearly) type-I-preserving then so, respectively, is $\bG\times \bH$.
\end{lemma}
\begin{proof}

  Consider an action of $\bG\times \bH$ on a type-I group $\bM$ (linear or not, depending on which branch of the statement we are considering). That action restricts to either
  \begin{equation*}
    \bG\le \bG\times \bH\text{ or }\bH\le \bG\times \bH,
  \end{equation*}
  and upon extending the $\bH$-action to $\bM\rtimes \bG$ by acting trivially on the $\bG$ factor we have an isomorphism
  \begin{equation}\label{eq:mgh}
    \bM\rtimes(\bG\times\bH) \cong (\bM\rtimes \bG)\rtimes \bH.
  \end{equation}
  The initial semidirect product $\bM\rtimes \bG$ is type-I by the assumption on $\bG$ (either of the two), whence \Cref{eq:mgh} is type-I because $\bH$ is assumed type-I-preserving.
\end{proof}

The distinction between type-I preservation and its linear counterpart is already visible for discrete groups: this is obvious from \Cref{cor:disc}, say, but we prove a more general version thereof.

\begin{theorem}\label{th:disc-by-cpct}
  Let $\bG$ be a second-countable locally compact group fitting into an exact sequence
  \begin{equation}\label{eq:dgk}
    \{1\}\to \bD\to \bG\to \bK\to \{1\},
  \end{equation}
  with $\bD$ discrete and $\bK$ compact.
  \begin{enumerate}[(1)]
  \item\label{item:12} $\bG$ is type-I-preserving if it is compact, and the converse holds if $\bK$ is Lie.
  \item\label{item:13} $\bG$ is linearly type-I-preserving if it has a cocompact, discrete, normal abelian subgroup $\bA$ so that the orbit of every infinite-order element of $\widehat{\bA}$ under the conjugation action is infinite.
  \item\label{item:haschar} Conversely, if some cocompact normal, discrete subgroup $\bA\trianglelefteq \bG$ has an infinite-order character with finite orbit under the conjugation $\bG$-action then $\bG$ is not linearly type-I-preserving. 
  \item\label{item:14} $\bG$ is of type I if $\bD$ is virtually abelian, and the converse holds if $\bK$ is Lie.
  \end{enumerate}
\end{theorem}

\begin{proof}
  Part \Cref{item:14} follows immediately from \Cref{cor:ext} and \Cref{pr:cpctliequot} in conjunction with the previously-cited theorem of Thoma characterizing discrete type-I groups, so we focus on the two other claims.

  {\bf ($\Leftarrow$) part \Cref{item:12}.} Immediate from \Cref{le:t1impl}.
      
  {\bf ($\Rightarrow$) part \Cref{item:12}: $\bD$ is virtually abelian.} $\bG$ itself is type-I in both cases by \Cref{re:ist1}, hence so is its cocompact subgroup $\bD$ (\Cref{pr:cpctliequot}). We conclude via Thoma's theorem.
  
  {\bf ($\Rightarrow$) part \Cref{item:12}: $\bD$ can be assumed abelian.} Given that $\bD$ is virtually abelian, it has a finite-index abelian {\it characteristic} subgroup \cite[Lemma 21.1.4]{km} (i.e. one invariant under all automorphisms of $\bD$). In particular such a group will be normal in $\bG$, and we can substitute it for $\bD$ in the statement. 

  This abelianness assumption on $\bD$ is in force throughout the rest of the proof, and in \Cref{item:13} the sought-after $\bA$ will be $\bD$. In other words, the goal is now to show that
  \begin{itemize}
  \item $\bD$ is finite in \Cref{item:12};
  \item while in \Cref{item:13}, the infinite-order elements of $\widehat{\bD}$ have infinite orbits under the conjugation action by $\bK$ (or $\bG$).
  \end{itemize}
    
  {\bf ($\Rightarrow$) part \Cref{item:12}: finitely-generated-kernel surjections onto direct products.} $\bD$ being abelian, the exact sequence \Cref{eq:dgk} corresponds to a cohomology class in $H^2(\bK,\bD)$, where the cohomology groups are those introduced on \cite[p.43]{moore-ext} (see also \cite[top of p.44, before \S 2]{moore-ext} for a sketch of the extension-cohomology correspondence and \cite[Theorem 7.8]{vrd} for a detailed treatment of {\it central} extensions).
 
  Because $\bK$ is compact, \cite[Lemma 2.2]{moore-ext} shows that the cohomology class in question is in the image of
  \begin{equation*}
    H^2(\bK,\bD_1)\to H^2(\bK,\bD)
  \end{equation*}
  attached to the inclusion $\bD_1\le \bD$ of some $\bK$-invariant finitely-generated subgroup. But this means that the cohomology class attached to the extension
  \begin{equation*}
    \{1\}\to \bD/\bD_1\to \bG/\bD_1\to \bK\to \{1\}
  \end{equation*}
  is trivial, i.e. this latter extension splits:
  \begin{equation*}
    \bG/\bD_1\cong (\bD/\bD_1)\rtimes \bK,
  \end{equation*}
  We retain this setup of an extension
  \begin{equation}\label{eq:dd1k}
    \{1\}\to \bD_1\to \bG\to (\bD/\bD_1)\rtimes \bK\to \{1\}
  \end{equation}
  for finitely-generated abelian $\bD_1$ throughout, with the product $(\bD/\bD_1)\rtimes \bK$ assumed direct in case \Cref{item:13}.
  
    {\bf ($\Rightarrow$) part \Cref{item:12}: semidirect products.} That is, we consider the particular case where \Cref{eq:dgk} splits, and thus
  \begin{equation*}
    \bG\cong \bD\rtimes \bK;
  \end{equation*}
  the goal will be to assume $\bD$ infinite and derive a contradiction.

  Consider the countable group
  \begin{equation*}
    \bF:=\{\text{finitely-supported functions }\bD\to \bZ/2\}
  \end{equation*}
  with its obvious additive structure ($\bZ/2$ is there only to fix ideas: it can be any fixed finite abelian group for our purposes). $\bD$ acts on $\bF$ by translation:
  \begin{equation*}
    (gf)(g') = f(g'g),\ \forall g,g'\in \bD,\ \forall f\in \bF. 
  \end{equation*}
  That translation action extends to one by $\bG\cong \bD\rtimes \bK$, and the semidirect product $(\bF\rtimes \bD)\rtimes \bK$ is type-I by assumption. But then so is the countable discrete group $\bF\rtimes \bD$ (being normal cocompact: \Cref{pr:cpctliequot}), and Thoma's theorem provides a cofinite abelian subgroup
  \begin{equation*}
    \bA\le \bF\rtimes \bD.
  \end{equation*}
  We can now obtain our contradiction: $\bA$ has finite index, so the intersections
  \begin{equation*}
    \bF\cap \bA\text{ and }\bD\cap \bA
  \end{equation*}
  must both be infinite and in particular non-trivial. But this contradicts the abelianness of $\bA$, since any non-trivial $g\in \bD\cap \bA$ will translate the support of any non-trivial $f\in \bF\cap \bA$, so that $gfg^{-1}\ne f$.  

  {\bf ($\Rightarrow$) part \Cref{item:12}: conclusion.} We have already disposed of the semidirect-product quotient in \Cref{eq:dd1k}, which we now know is compact; we can thus absorb the $\bD/\bD_1$ kernel into $\bK$:
  \begin{equation}\label{eq:dgak-simple}
    \{1\}\to \bD_1\to \bG\to \bK\to \{1\}
  \end{equation}

  $\bD_1$ is finitely-generated abelian, so its finiteness amounts to the vanishing of its free abelian quotient $\bZ^d$, $d\in \bZ_{\ge 0}$ by its torsion. To that end, we will assume that in fact $\bD_1\cong \bZ^d$, $d\ge 1$ and show we cannot even have {\it linear} type-I preservation.

  The compact group $\bK$ acts on $\bD_1\cong \bZ^d$ through a finite quotient because
  \begin{equation*}
    \mathrm{Aut}(\bZ_1^d)\cong GL(d,\bZ)
  \end{equation*}
  is discrete. Some finite-index subgroup of $\bK$ thus acts trivially on $\bD_1$, so for our purposes we may as well assume that \Cref{eq:dgak-simple} is central (since the passage to a finite-index subgroup will not affect linear type-I preservation \Cref{le:changegp}).
  
  We can now further surject $\bD_1\to \bZ$ (since the kernel of such a surjection is central in $\bG$), so that \Cref{eq:dgak-simple} is a central extension
  \begin{equation}\label{eq:zgak}
    \{1\}\to \bZ\to \bG\to \bK\to \{1\}.
  \end{equation}
  It again corresponds to a cohomology class in $H^2(\bK,\bZ)$, and the discussion on \cite[p.61]{moore-ext} provides an identification
  \begin{equation*}
    H^2(\bK,\bZ)\cong H^1(\bK,\bS^1)\cong \mathrm{Hom}(\bK,\bS^1). 
  \end{equation*}
  Running through that argument, said identification recovers a central extension
  \begin{equation*}
    \{1\}\to \bZ\to \bullet\to \bK\to \{1\}
  \end{equation*}
  from a morphism $\varphi:\bK\to \bS^1$ as a pullback
  \begin{equation*}
    \begin{tikzpicture}[auto,baseline=(current  bounding  box.center)]
      \path[anchor=base] 
      (0,0) node (ll) {$\{1\}$}
      +(2,0) node (l) {$\bZ$}
      +(4,.5) node (u) {$\bullet$}
      +(4,-.5) node (d) {$\bR$}
      +(6,.5) node (ur) {$\bK$}
      +(6,-.5) node (dr) {$\bS^1$}
      +(8,0) node (rr) {$\{1\}$,}
      ;
      \draw[->] (ll) to[bend left=0] node[pos=.5,auto] {$\scriptstyle $} (l);
      \draw[->] (l) to[bend left=6] node[pos=.5,auto] {$\scriptstyle $} (u);
      \draw[->] (l) to[bend right=6] node[pos=.5,auto,swap] {$\scriptstyle $} (d);
      \draw[->] (u) to[bend left=0] node[pos=.5,auto] {$\scriptstyle $} (ur);
      \draw[->] (d) to[bend left=0] node[pos=.5,auto] {$\scriptstyle $} (dr);
      \draw[->] (u) to[bend left=0] node[pos=.5,auto] {$\scriptstyle $} (d);
      \draw[->] (ur) to[bend left=0] node[pos=.5,auto] {$\scriptstyle \varphi$} (dr);
      \draw[->] (ur) to[bend left=6] node[pos=.5,auto] {$\scriptstyle $} (rr);
      \draw[->] (dr) to[bend right=6] node[pos=.5,auto,swap] {$\scriptstyle $} (rr);
    \end{tikzpicture}
  \end{equation*}
  where the bottom extension results from the usual exponential map
  \begin{equation*}
    \bR\ni t\mapsto \exp(2\pi i t)\in \bS^1.
  \end{equation*}
  It follows that we can replace $\bG$ by its quotient obtained upon substituting the image $\varphi(\bK)\subset \bS^1$ for $\bK$; the problem has thus been further reduced to abelian $\bK$ (and in fact a closed subgroup of the circle).

  Now, since \Cref{eq:zgak} is a central extension with abelian quotient, the commutator
  \begin{equation*}
    \bG\times \bG\ni (g_1,g_2)\mapsto [g_1,g_2]:=g_1g_2g_1^{-1}g_2^{-1}\in \bZ
  \end{equation*}
  induces a continuous bilinear map $\bK^2\to \bZ$, trivial because $\bK$ is compact. $\bG$ is thus abelian, and we can fall back on the already-settled case of LCA groups to conclude that indeed the group $\bD_1$ of \Cref{eq:dgak-simple} is finite. This finishes the proof of \Cref{item:12}.

{\bf Part \Cref{item:13}.} Suppose $\bA\trianglelefteq \bG$ is as in the statement. The plan will be to show that $\bA$ operates on $V$ through a finite quotient. Assuming this for now, we can proceed as follows.
  \begin{itemize}
  \item Because $\pi|_{\bA}$ factors through a finite group, $\bG$ acts on $V^*\cong \widehat{V}$ with compact orbits;
  \item and hence the {\it regular embedding} condition of \cite[p.186, Definition]{mack-unit} is satisfied (e.g. by \cite[Theorem 1]{glm});
  \item so \cite[Theorem 3.12]{mack-unit} applies (in particular the last paragraph of that statement, on semidirect products with abelian kernel) to conclude that $V\rtimes \bA$ is of type I.
  \end{itemize}
  Since
  \begin{equation*}
    V\rtimes \bA\le V\rtimes \bG
  \end{equation*}
  is cocompact and normal larger group must also be of type I by \Cref{cor:ext}, and we are done. It thus remains to prove the claim that $\pi(\bA)\subset GL(V)$ is finite.

  The representation $\pi:\bG\to GL(V)$ factors through the quotient
  \begin{equation*}
    \begin{tikzpicture}[auto,baseline=(current  bounding  box.center)]
      \path[anchor=base] 
      (0,0) node (ll) {$\{1\}$}
      +(2,.5) node (lu) {$\bA$}
      +(2,-.5) node (ld) {$\bA_0$}
      +(4,.5) node (mu) {$\bG$}
      +(4,-.5) node (md) {$\bG_0$}
      +(6,0) node (r) {$\bG/\bA$}
      +(8,0) node (rr) {$\{1\}$,}
      ;
      \draw[->] (ll) to[bend left=6] node[pos=.5,auto] {$\scriptstyle $} (lu);
      \draw[->] (ll) to[bend right=6] node[pos=.5,auto] {$\scriptstyle $} (ld);
      \draw[->] (lu) to[bend left=0] node[pos=.5,auto] {$\scriptstyle $} (mu);
      \draw[->] (ld) to[bend right=0] node[pos=.5,auto,swap] {$\scriptstyle $} (md);
      \draw[->] (mu) to[bend left=6] node[pos=.5,auto] {$\scriptstyle $} (r);
      \draw[->] (md) to[bend right=6] node[pos=.5,auto] {$\scriptstyle $} (r);
      \draw[->] (r) to[bend left=0] node[pos=.5,auto] {$\scriptstyle $} (rr);
      \draw[->] (lu) to[bend left=0] node[pos=.5,auto] {$\scriptstyle $} (ld);
      \draw[->] (mu) to[bend left=0] node[pos=.5,auto] {$\scriptstyle $} (md);
    \end{tikzpicture}
  \end{equation*}
  where
  \begin{equation*}
    \bA_0:=\bA/\ker(\pi|_{\bA}). 
  \end{equation*}
  The closure $\overline{\pi(\bA_0)}$ in $GL(V)$ is an abelian Lie group, and hence of the form
  \begin{equation}\label{eq:rtd}
    \overline{\pi(\bA_0)}\cong \bR^n\times \bT^m\times \bD,\quad \bD\text{ finitely-generated discrete abelian}
  \end{equation}
  (e.g. by \cite[Theorem 4.2.4]{de}). The compact group 
  \begin{equation*}
    \bG_0/\bA_0 \cong \bG/\bA =: \bK
  \end{equation*}
  acts by conjugation on $\bA_0$, and that action extends to a continuous one on \Cref{eq:rtd}. Since the automorphism group of the latter Lie group is Lie again \cite[Theorem 2]{hoch-auto}, the action of $\bK$ on $\overline{\pi(\bA_0)}$ factors through a compact Lie quotient $\widetilde{\bK}$ of $\bK$.
  
  At the same time, the action of the compact group $\bK$ on the discrete group $\bA_0$ has finite orbits; it follows that \Cref{eq:rtd} has a dense set $\pi(\bA_0)\cong \bA_0$ of elements with finite orbits under the action of the compact Lie group $\widetilde{\bK}$, whence that action must factor through a finite group. In conclusion, the centralizer of $\pi(\bA_0)\cong \bA_0$ in $\bG_0$ is cofinite.

  
  The infinite-orbit assumption then implies that {\it all} characters of $\bA_0$ are of finite order, and hence so are those of \Cref{eq:rtd}. This, in turn, implies that the continuous Cartesian factors $\bR^n$ and $\bT^m$ in that product are trivial, and $\bD$ is finite.


  {\bf Part \Cref{item:haschar}.} Consider a character $\chi:\bA\to \bS^1$ as in the statement. Since it has finite orbit under the $\bG$-action, passage to a cofinite subgroup (which does not affect linear type-I preservation by \Cref{le:changegp}) allows us to assume that $\chi$ is in fact $\bG$-invariant.

  We can now further annihilate the (normal) subgroup of $\bA$ generated by commutators
  \begin{equation*}
    [g,a] = gag^{-1}a^{-1},\ g\in \bG,\ a\in \bA
  \end{equation*}
  thus reducing to the case of central $\bA$. The character $\chi$ descends to the respective quotient of $\bA$ (by $\bG$-invariance), so it is still present under the centrality assumption.

  The central extension
  \begin{equation}\label{eq:agk}
    \{1\}\to \bA\to \bG\to \bK\to \{1\}
  \end{equation}
  corresponds to a cohomology class in
  \begin{equation}\label{eq:h2lim}
    H^2(\bK,\bA)\cong \varinjlim_i H^2(\bK_i,\bA),
  \end{equation}
  where
  \begin{equation*}
    \bK\cong \varprojlim \bK_i
  \end{equation*}
  expresses $\bK$ as an inverse limit of compact Lie groups \cite[\S 4.6, Theorem]{mz} and \Cref{eq:h2lim} is \cite[Theorem 2.3 (2)]{moore-ext}. It thus follows that the kernel of some $\bK\to \bK_i$ pulls down as a normal subgroup of $\bG$, and quotienting by it we can assume that $\bK$ is Lie, and, for good measure, also connected (the connected component has finite index).
  
  A quotient of $\bK$ by some finite central subgroup $\bF\le \bK$ is a product of a torus and simple compact factors \cite[Theorem 9.24]{hm3}. We can first regroup $\bF$ together with $\bA$ into a central extension
  \begin{equation*}
    \{1\}\to \bA\to \bullet\to \bF\to \{1\},
  \end{equation*}
  which becomes abelian after quotienting by a finite subgroup of $\bA$: because $\bA$ is central and $\bF$ abelian the commutator on $\bullet$ descends to a bilinear map $\bF\times \bF\to \bA$, which must take values in a finite subgroup because $\bF$ is finite.

  It follows, then, that we can further assume that
  \begin{equation*}
    \bK\cong \bT^d\times \prod_{i=1}^n \bS_i,\quad \bS_i\text{ simple, compact, connected Lie groups}.
  \end{equation*}
  Semisimple compact Lie groups have compact universal covers \cite[Theorem 6.6]{hm3}, so the $\bS_i$-components of a 2-cocycle representing \Cref{eq:agk} as an element of
  \begin{equation*}
    H^2(\bK,\bA)\cong H^2(\bT^d,\bA)\times \prod_{i=1}^n H^2(\bS_i,\bA)
  \end{equation*}
  will take values in some finite subgroup of $\bA$. Quotienting again, the semisimple factors fall out as subgroups, and can be annihilated. \Cref{eq:agk} has become
  \begin{equation*}
    \{1\}\to \bA\to \bG\to \bT^d\to \{1\}.
  \end{equation*}
  Note furthermore that we can assume $\bG$ is abelian, by the same commutator bilinear map argument employed above.
  
  The cocycle corresponding to this extension factors through some finitely-generated abelian subgroup $\bA_f\le \bA$ \cite[Lemma 2.2]{moore-ext}; it follows that $\bG$ contains 
  \begin{equation*}
    \{1\}\to \bA_f\to \bG_f\to \bT^d\to \{1\}.
  \end{equation*}
  as an open subgroup, with corresponding quotient
  \begin{equation*}
    \bG/\bG_f\cong \bA/\bA_f
  \end{equation*}
  There are now two cases to consider:
  \begin{enumerate}[(a)]
  \item\label{item:aafa} {\bf The quotient $\bA/\bA_f$ has an infinite-order character.} The quotient $\bA/\bA_f$ then acts on $\bC$ via some infinite-order character. The image through that character must then be an infinite countable subgroup of the circle, so it is not locally closed in $\bC$ (i.e. relatively open in its closure). That $\bC\rtimes (\bA/\bA_f)$ cannot be of type I follows from \cite[p.110, Corollary to Theorem 9]{am} (that result only discusses ``non-transitive quasi-orbits'', but by \cite[Theorem 1]{glm} we have these precisely under the stated condition).
  \item\label{item:aafb} {\bf The quotient $\bA/\bA_f$ has no infinite-order characters.} We will see later, in the course of the proof of \Cref{cor:disc}, that this means precisely that it has {\it bounded order} \cite[\S 8]{kap}: it is torsion, with a uniform bound on the orders of its elements.

    It is not difficult to see then that some finitely-generated extension of $\bA_f$ in $\bA$ is {\it pure} \cite[\S 7]{kap}, whence \cite[Theorems 5 and 6]{kap} imply that that newly-enlarged $\bA_f$ splits off as a direct summand:
    \begin{equation*}
      \bA = \bA_f\oplus \bB,\quad \bB\text{ of bounded order}.
    \end{equation*}
    Quotient out $\bB$ and some cyclic summands of (the finitely-generated) $\bA_f$ to finally bring $\bG$ to the form
    \begin{equation*}
      \{1\}\to \bZ\to \bG \to \bT^d\to \{1\}.
    \end{equation*}
    The identification
    \begin{equation*}
      H^2(\bT,\bZ)\cong \mathrm{Hom}(\pi_1(\bT^d),\bZ)\cong \mathrm{Hom}(\bZ^d,\bZ)
    \end{equation*}
    mentioned in the proof of \cite[II, Proposition 1.2]{moore-ext} makes it clear that $\bG$ either splits as $\bZ\oplus \bT^d$ or contains $\bR\oplus \bT^{d-1}$ as a cofinite subgroup. Either way, it fails to be linearly type-I-preserving:
    \begin{itemize}
    \item $\bZ$ isn't, as in \Cref{item:aafa}, because it is discrete and has an infinite-order character,
    \item and $\bR$ isn't by {\it Mautner's example} \cite[pp.137-138]{am}, of a non-type-I group of the form $\bC^2\rtimes \bR$ for an appropriately-chosen linear (unitary, in fact) representation of $\bR$ on $\bC^2$.
    \end{itemize}    
  \end{enumerate}

  This concludes the proof of the result as a whole.  
\end{proof}


As a consequence, when the compact quotient $\bK$ in \Cref{eq:dgk} is absent we have a characterizations for (linear) type-I preservation for discrete groups. Recall that a group is {\it of bounded order} if there is some $n$ such that all elements are of order $\le n$ (\cite[\S 8]{kap}, for instance).

\begin{corollary}\label{cor:disc}
    A discrete countable group $\bG$ is
  \begin{enumerate}[(1)]
  \item\label{item:6} type-I-preserving if and only if it is finite;
  \item\label{item:7} linearly type-I-preserving if and only if it has a finite-index abelian subgroup $\bA\le \bG$ satisfying any of the following equivalent conditions.
    \begin{enumerate}[(a)]
    \item\label{item:15} $\bA$ is of bounded order.
    \item\label{item:16} $\bA$ is a direct sum
      \begin{equation*}
        \bA\cong \bigoplus_{n_i}\bZ/n_i,\quad (n_i)_i\text{ bounded}. 
      \end{equation*}
    \item\label{item:17} The compact Pontryagin dual $\widehat{\bA}$ is a product
      \begin{equation*}
        \widehat{\bA}\cong \prod_{n_i}\bZ/n_i,\quad (n_i)_i\text{ bounded}. 
      \end{equation*}
    \item\label{item:18} The dual $\widehat{\bA}$ is of bounded order.
    \item\label{item:19} The dual $\widehat{\bA}$ is torsion. 
    \end{enumerate}    
  \end{enumerate} 
\end{corollary}
\begin{proof}
  The two claims are what parts \Cref{item:12} and respectively \Cref{item:13} of \Cref{th:disc-by-cpct} specialize to upon setting $\bK=\{1\}$, modulo the mutual equivalence of the various conditions in \Cref{item:7}: \Cref{th:disc-by-cpct} \Cref{item:13} requires that the compact group $\widehat{\bA}$ be torsion, i.e. condition \Cref{item:19} of the present statement. We briefly sketch why the others amount to the same constraint (a well-known result).

  Note that
  \begin{itemize}
  \item \Cref{item:15} and \Cref{item:18} are mutual Pontryagin duals and hence equivalent;
  \item similarly for \Cref{item:17} and \Cref{item:16};
  \item the latter pair clearly implies the former;
  \item whereas the converse, in the form \Cref{item:15} $\Rightarrow$ \Cref{item:16}, say, is \cite[Theorem 6]{kap};
  \item so that the first four conditions are mutually equivalent, and they clearly imply the fifth. 
  \end{itemize}
  It thus remains to recall that compact abelian torsion groups are automatically of bounded order: e.g. \cite[p.70]{mor-pont}.
\end{proof}

An analogue to \Cref{cor:disc} obtains upon substituting abelianness for discreteness. We will make repeated use of the main structure theorem for locally compact abelian (LCA, for short) groups, \cite[Theorem 4.2.1]{de}: an LCA group $\bG$ decomposes as
\begin{equation}\label{eq:krd}
  \bG\cong \bG_1\times \bR^d
\end{equation}
with $\bG_1$ admitting a compact open subgroup.

The number $d$ is an invariant attached canonically to $\bG$ that we will refer to as the {\it characteristic index} (following \cite[Definition preceding Lemma 3.15]{iw}, which introduces an overlapping notion).

\begin{definition}\label{def:vfr}
  A locally compact abelian group $\bG$ with an open subgroup \Cref{eq:krd} is {\it vector-group-free} or {\it vector-free} or {\it vector-less} if its characteristic index $d$ vanishes. 
\end{definition}

\begin{proposition}\label{pr:lca}
  A second-countable locally compact abelian group $\bG$ is
  \begin{enumerate}[(1)]
  \item\label{item:10} type-I-preserving if and only if it is compact;
  \item\label{item:11} linearly type-I-preserving if and only if it has an open compact subgroup $\bK\le \bG$ with bounded-order quotient $\bG/\bK$.
  \end{enumerate}
\end{proposition}
\begin{proof}
  A number of partial results will coalesce into the desired result. Throughout, we consider a decomposition \Cref{eq:krd}.
  
  {\bf (Linearly) type-I-preserving implies vector-less.} \Cref{le:quot} implies that the quotient $\bG\to \bR^d$ must itself be (linearly) type-I-preserving, along with its quotient $\bR^d\to \bR$ if $d\ge 1$. This, though, is not the case:
  \begin{itemize}
  \item represent $\bR$ on some 2-dimensional complex vector space $V$ via 
    \begin{equation*}
      \bR\ni t\mapsto \mathrm{diag}\left(e^{2\pi i t}, e^{2\pi i \lambda t}\right) \in U(V)
    \end{equation*}
    for irrational $\lambda$, as in the {\it Mautner example} denoted by $M_5$ on \cite[pp.137-138]{am} (also \cite[\S 6.8, 1.]{folland});
  \item and conclude as in loc.cit., from \cite[p.110, Corollary to Theorem 9]{am}, that the semidirect product $V\rtimes \bR$ cannot be of type I.
  \end{itemize}

  {\bf Claim \Cref{item:10}.} That compactness implies type-I-preservation follows from \Cref{le:t1impl}, so we are interested in the converse. Vector-less-ness implies that $\bG$ has an open compact subgroup; the quotient is discrete and type-I-preserving, hence finite by \Cref{cor:disc} \Cref{item:6}, and we are done.

  {\bf Claim \Cref{item:11}: ($\Rightarrow$).} Vector-less-ness again implies the existence of a compact open subgroup $\bK$, and the discrete quotient $\bG/\bK$ must be of bounded order by \Cref{le:quot} and \Cref{cor:disc} \Cref{item:7}.

  {\bf Claim \Cref{item:11}: ($\Leftarrow$).} Consider a finite-dimensional $\bG$-representation
  \begin{equation*}
    \pi:\bG\to GL(V). 
  \end{equation*}
  The type-I character of $V\rtimes \bG$ will follow from \cite[Theorem 3.12]{mack-unit} via \cite[Theorem 1]{glm} once we argue that all orbits in finite-dimensional $\bG$-representations are locally closed.
  
  To that end, note that in fact all such orbits are {\it compact} (and hence closed). Indeed, let $\bK\le \bG$ be a compact open subgroup with bounded-order attached quotient $\bM:=\bG/\bK$. The closed images
  \begin{equation*}
    \pi(\bK)\subset GL(V)\text{ (compact hence automatically closed) }\text{ and }\overline{\pi(\bG)}\subset GL(V)
  \end{equation*}
  are both closed abelian Lie subgroups with finitely many components. It follows that for every $n$ the group of elements of order $\le n$ in the Lie group $\overline{\pi(\bG)}/\pi(\bK)$ is finite, and hence the image of the bounded-order group $\bM$ in that quotient is finite.

  All in all, this implies that the image $\pi(\bG)$ ({\it without} taking the closure) is a finite extension of the compact Lie group $\pi(\bK)$ and hence again compact Lie. In particular, its orbits in $V$ (or $V^*$) are compact. 
\end{proof}

\Cref{pr:lca} \Cref{item:10} goes through under the weaker hypothesis of nilpotence (as opposed to abelianness). We first need the following auxiliary observation. 

\begin{proposition}\label{pr:nosurj}
  A locally compact nilpotent group with no quotients that are either infinite discrete abelian or vector-groups is compact.
\end{proposition}
\begin{proof}
The (closed) lower central series \cite[Definition following Lemma 5.30]{rot-gp}
  \begin{equation}\label{eq:lser}
    \{1\}\subset \cdots \subset \overline{[\bG,[\bG,\bG]]} \subset\overline{[\bG,\bG]}\subset \bG
  \end{equation}
  of our nilpotent group $\bG$ is finite. We proceed by induction on its length.

  The base case of abelian $\bG$ follows from the product decomposition \Cref{eq:krd}, so suppose the claim holds for groups with shorter central series. To tackle the induction step, it will be enough to restrict attention to the following simplified setup:
  \begin{itemize}
  \item $\bG$ has a central subgroup $\bM$;
  \item with $\bK:=\bG/\bM$ compact (i.e. $\bM\le \bG$ is cocompact).
  \end{itemize}
  Being abelian, $\bM$ decomposes as
  \begin{equation*}
    \bM\cong \bM_1\times \bR^d
  \end{equation*}
  with $\bM_1$ having a compact open subgroup $\bL$ (by \cite[Theorem 4.2.1]{de}, as in \Cref{eq:krd}), and the central extension
  \begin{equation}\label{eq:mgk}
    \{1\}\to \bM\to \bG\to \bK\to \{1\}
  \end{equation}
  corresponds to a cohomology class in
  \begin{equation*}
    H^2(\bK,\bM)\cong H^2(\bK,\bM_1\times \bR^d)\cong H^2(\bK,\bM_1)\oplus H^2(\bK,\bR^d).
  \end{equation*}
  The cohomology group $H^2(\bK,\bR^d)$ vanishes (for any compact $\bK$; \cite[Theorem 2.3]{moore-ext}), so the $\bR^d$ summand of $\bM$ splits off globally. In particular $\bG$ surjects onto $\bR^d$, so $d=0$. By further quotienting out the compact open subgroup $\bL\le \bM$ (which is central in $\bG$), we can assume $\bM$ itself is discrete.

  Now, let $\bH\le \bG$ be the closed subgroup such that $\bH/\bM\le \bG/\bM$ is the center. The commutator map
  \begin{equation*}
    \bG\times \bH\ni (g,h)\mapsto [g,h]:=ghg^{-1}h^{-1}\in \bM
  \end{equation*}
  is continuous, bilinear (because $\bH$ is central modulo $\bM$, which is central period), and factors through the compact space $\bG/\bM\times \bH/\bM$. It follows that that map takes values in a finite subgroup of $\bM$; upon quotienting it out, we can assume that $\bH$ itself is central in $\bG$.
  
  But now we have shortened the central series, and can conclude by induction.
\end{proof}

This now easily implies the announced

\begin{proposition}\label{pr:nilp}
  A second-countable locally compact nilpotent group $\bG$ is type-I-preserving if and only if it is compact.
\end{proposition}
\begin{proof}

  Immediate from \Cref{pr:nosurj}, given that type-I preservation passes to quotients and neither vector groups nor infinite discrete nilpotent groups have it (\Cref{pr:lca} \Cref{item:10}).
\end{proof}

\subsection{Some totally disconnected examples}\label{subse:totdisc}

We gather some examples meant to address various questions that the preceding results might raise naturally.

First, prompted by \Cref{cor:disc} \Cref{item:7}, one might wonder whether the discrete group $\bD$ in \Cref{eq:dgk} must be of bounded order. The answer is negative.

\begin{example}\label{ex:notbdd}
  Let
  \begin{equation*}
    \bD:=\bigoplus_{\text{primes }p}\bZ/p,
  \end{equation*}
  take for $\bK$ its (compact) automorphism group
  \begin{equation}\label{eq:autgp}
    \bK=\mathrm{Aut}(\bD)\cong \prod_{p}\mathrm{Aut}(\bZ/p)\cong \prod_p \bZ/(p-1),
  \end{equation}
  and assemble the extension \Cref{eq:dgk} out of the semidirect product $\bG:=\bD\rtimes \bK$.

  Evidently, $\bD$ is not of bounded order. On the other hand, it must be linearly type-I-preserving by \Cref{th:disc-by-cpct} \Cref{item:13}: the infinite-order elements of the compact dual
  \begin{equation*}
    \widehat{\bD}\cong \prod_p \bZ/p
  \end{equation*}
  are precisely the elements with non-zero components in infinitely many $\bZ/p$ factors, and such elements have infinite orbits under \Cref{eq:autgp}. 
\end{example}

\begin{remark}\label{re:notnilp}
  Although obviously solvable, the group $\bG$ of \Cref{ex:notbdd} is not nilpotent, and cannot be: in \Cref{th:disc-by-cpct} \Cref{item:13}, under the additional assumption that $\bG$ is nilpotent, we can filter the discrete, normal, abelian subgroup $\bA\trianglelefteq \bG$ by
  \begin{equation*}
    \{1\}=\bA_{-1}\le \bA_0\le \cdots\le \bA_n=\bA
  \end{equation*}
  so that each $\bA_{n+1}/\bA_n$ is central in $\bG/\bA_n$. That centrality then ensures that $\bK$ acts trivially on $\bA_{n+1}/\bA_n$, and $\bA$ can be shown to have bounded order by induction on $n$.
\end{remark}

\begin{remark}
  Incidentally, \Cref{ex:notbdd} also shows that linear type-I preservation does not descend along normal, cocompact embeddings: we argued in the example that $\bG$ is linearly type-I-preserving, whereas the normal, cocompact, discrete abelian group $\bD\trianglelefteq \bG$ cannot be (by \Cref{pr:lca} \Cref{item:11}).

  The cofiniteness assumption of \Cref{le:changegp} \Cref{item:9} is thus essential.
\end{remark}

Next, regarding \Cref{pr:lca} \Cref{item:11}, note that the compact and discrete pieces may well both be present.

\begin{example}\label{ex:cpctbot}
  Consider the abelian extension
  \begin{equation*}
    \{1\}\to (\bZ/n)^{\aleph_0}\to \bG\to (\bZ/n)^{\oplus\aleph_0}\to \{1\}
  \end{equation*}
  (so that the quotient is a discrete direct {\it sum} and the kernel is a compact direct {\it product}) associated via \cite[Theorem 7.8]{vrd} to the image through
  \begin{equation*}
    H^2\left((\bZ/n)^{\oplus\aleph_0},(\bZ/n)^{\oplus\aleph_0}\right)\to H^2\left((\bZ/n)^{\oplus\aleph_0},(\bZ/n)^{\aleph_0}\right)
  \end{equation*}
  of the 2-cocycle giving the obvious extension
  \begin{equation*}
    \{1\}\to (\bZ/n)^{\oplus\aleph_0}\to (\bZ/n^2)^{\oplus\aleph_0}\to (\bZ/n)^{\oplus\aleph_0}\to \{1\}.
  \end{equation*}
  $\bG$ is linearly type-I-preserving by \Cref{pr:lca} \Cref{item:11}, but this time the compact group is a kernel and cannot be a quotient, as in \Cref{ex:infheis}.
\end{example}

On the other hand, {\it nilpotent} examples covered by \Cref{th:disc-by-cpct} \Cref{item:13} show that such a group need not have compact open normal subgroups.

\begin{example}\label{ex:infheis}
  Consider
  \begin{itemize}
  \item the compact group
    \begin{equation*}
      \bK:=(\bZ/n)^{\aleph_0};
    \end{equation*}    
  \item  the discrete groups
    \begin{equation*}
      \bA = \bD := (\bZ/n)^{\oplus \aleph_0};
    \end{equation*}
  \item the (continuous) bilinear morphism
    \begin{equation*}
      \varphi:\bA\times \bK\to \bD
    \end{equation*}
    pairing the individual $\aleph_0$-indexed summands via the ring multiplication
    \begin{equation*}
      \bZ/n\times \bZ/n\to \bZ/n.
    \end{equation*}
  \end{itemize}
  This data gives rise to a nilpotent central extension
  \begin{equation*}
    \{1\}\to \bD\to \bG\to \bA\times \bK\to \{1\}
  \end{equation*}
  whereby the commutator morphism
  \begin{equation*}
    \bG^{\times 2}\ni (g_1,g_2)\mapsto [g_1,g_2] := g_1g_2g_1^{-1}g_2^{-1}
  \end{equation*}
  descends to $\varphi$. More formally, this is the central extension associated via \cite[Theorem 7.8]{vrd} to the unique 2-cocycle
  \begin{equation*}
    m:(\bA\times \bK)^{\times 2}\to \bD
  \end{equation*}
  that restricts to $\varphi$ on $\bA\times \bK$ and is trivial on
  \begin{equation*}
    \bA^2,\quad \bK^2\quad\text{and}\quad \bK\times \bA. 
  \end{equation*}
  The group $\bG$ alternatively fits into (non-central) extensions
  \begin{equation}\label{eq:dabk}
    \{1\}\to \bD\times \bA\to \bG\to \bK\to \{1\}
  \end{equation}
  and
  \begin{equation*}
    \{1\}\to \bD\times \bK\to \bG\to \bA\to \{1\}.
  \end{equation*}
  Cast as \Cref{eq:dabk} it is proven linearly type-I-preserving by \Cref{th:disc-by-cpct} \Cref{item:13}, but $\bG$ is easily seen not to have a compact, open, normal subgroup $\bK_1\trianglelefteq \bG$: $\bG/\bK_1$ being open the intersection $\bK\cap \bK_1$ would have to be of finite index in $\bK$, but the construction makes it clear that operating on the finite-index subgroup
  \begin{equation*}
    \bK\cap \bK_1\le \bK
  \end{equation*}
  by conjugation with $\bA$ will produce elements ranging over a finite-index subgroup of the discrete center $\bD$, and hence force us out of the supposedly-normal subgroup $\bK_1\trianglelefteq \bG$.
\end{example}

\begin{remark}\label{re:cpctopenorm}
  The issue of whether a totally-disconnected locally compact group $\bG$ has compact, open, {\it normal} subgroups has received some attention in the literature: there are always, in a sense, ``enough'' such subgroups when $\bG$ is {\it compactly generated} \cite[Theorem]{wil-nilp}. The example following that result shows that compact generation is necessary, as does \Cref{ex:infheis} above.
\end{remark}

\Cref{ex:cpctbot,ex:infheis} neatly split off a compact from a discrete subquotient, fitting the group into one of the two patterns
\begin{itemize}
\item compact-by-discrete;
\item or discrete-by-compact. 
\end{itemize}

The examples can be combined to produce a nilpotent linearly type-I-preserving group that is of neither type: there is a discrete subquotient ``trapped'' between two a normal compact subgroup and a compact quotient.

\begin{example}\label{ex:sandwich}
  This time we denote by $\bD$ the abelian group $\bG$ of \Cref{ex:cpctbot}, so that we have an extension
  \begin{equation}\label{eq:zdz}
    \{1\}\to (\bZ/n)^{\aleph_0}\to \bD\to (\bZ/n)^{\oplus\aleph_0}\to \{1\}.
  \end{equation}
  Next, set
  \begin{equation*}
    \bK:=(\bZ/n^2)^{\aleph_0},\quad \bA:=(\bZ/n^2)^{\oplus\aleph_0}.
  \end{equation*}
  Finally, consider a pairing
  \begin{equation*}
    \varphi:\bA\times \bK\to (\bZ/n^2)^{\bigoplus \aleph_0}\subset \bD
  \end{equation*}
  where the left-hand arrow is as in \Cref{ex:infheis} and the right-hand map is the obvious inclusion. This data is again sufficient to recover a central extension
  \begin{equation*}
    \{1\}\to \bD\to \bG\to \bA\times \bK\to \{1\},
  \end{equation*}
  linearly type-I-preserving:
  \begin{itemize}
  \item the cocompact, abelian subgroup $\bullet\le \bG$ in
    \begin{equation*}
      \{1\}\to \bD\to \bullet\to \bA\to \{1\}
    \end{equation*}
    is linearly type-I-preserving by \Cref{pr:lca};
  \item hence so is $\bG$, by \Cref{le:changegp} \Cref{item:8}.
  \end{itemize}
  That $\bG$ has no normal compact open subgroups can be argued as in \Cref{ex:infheis}. On the other hand, we now no longer have closed, discrete, normal, cocompact groups either. In fact, cocompactness does not play much of a role: every closed, discrete, normal subgroup $\bE\trianglelefteq \bG$ is finite, as we now argue.

  The extension \Cref{eq:zdz} is easily seen to be {\it essential} \cite[Derfinition preceding Proposition 3.43]{rot}, in the sense that every subgroup of $\bD$ intersects the left-hand term. Because that term is compact, $\bE$ intersects it and hence $\bD$ along a finite group:
  \begin{equation}\label{eq:defin}
    |\bE\cap\bD|<\infty
  \end{equation}
  Next, if $\bE$ were infinite, it would have infinite image in $\bA\times \bK$ and hence in one of $\bA$ or $\bK$. If the former, operating on that infinite image by conjugation with $\bK$ would produce infinitely many elements of $\bD$, contradicting \Cref{eq:defin}. If the latter, act by conjugation with $\bA$ instead.

  Either way, we conclude that $\bE$ must be finite.
\end{example}

\subsection{Connected groups}

\begin{proposition}\label{pr:ss}
  Connected, semisimple, linear Lie groups are linearly type-I-preserving.
\end{proposition}
\begin{proof}
  The hypothesis implies that
  \begin{itemize}
  \item $\bG$ is the identity component of the group $\bH(\bR)$ of real points of a real-algebraic group $\bH$ (e.g. \cite[\S 2.4]{ragh});
  \item and upon complexifying, the classification (\cite[\S 20.3]{hmph}) of finite-dimensional representations of semisimple Lie algebras implies that our representation $\bG\to GL(n,\bR)$ is algebraic.
  \end{itemize}
  But then $\bR^n\rtimes \bG$ can once more be identified with the identity component of $\bH(\bR)$ for some real-algebraic $\bH$, and is thus of type I by \cite[Th\'eor\`eme 1]{dix-alg}.
\end{proof}

At the other end of the spectrum (from semisimplicity):

\begin{proposition}\label{pr:nilpcon}
  For a second-countable, locally-compact, connected nilpotent group $\bG$ the following conditions are equivalent.
  \begin{enumerate}[(a)]
  \item\label{item:20} $\bG$ is compact abelian.
  \item\label{item:21} $\bG$ is compact.
  \item\label{item:22} $\bG$ is linearly type-I-preserving.
  \item\label{item:23} $\bG$ does not surject onto $\bR$.
  \end{enumerate}
\end{proposition}
\begin{proof}
  {\bf \Cref{item:20} $\Rightarrow$ \Cref{item:21}:} obvious. 

  {\bf \Cref{item:21} $\Rightarrow$ \Cref{item:22}:} by \Cref{le:t1impl}.

  {\bf \Cref{item:22} $\Rightarrow$ \Cref{item:23}:} linear type-I preservation is inherited by quotients \Cref{le:quot} and $\bR$ does not have the property, by Mautner's example \cite[\S 6.8, 1.]{folland} used in the proof of \Cref{pr:lca}.

  {\bf \Cref{item:23} $\Rightarrow$ \Cref{item:20}.} This is general structure theory for locally compact nilpotent groups.

  Connected, locally compact nilpotent groups have unique maximal compact subgroups, and these are central: this is mentioned in \cite[\S 1.8]{ragh} for {\it Lie} groups, but connected locally compact groups surject onto Lie groups with arbitrarily small compact normal kernels \cite[\S 4.6, Theorem]{mz}, and the remark of \cite[\S 1.8]{ragh} extends to the present setting.

  Now, if $\bK\le \bG$ is the maximal compact subgroup, then $\bG/\bK$ is a successive extension of vector groups \cite[Theorem 13]{iw}, and it follows that $\bG/\bK$ surjects onto $\bR$ unless it is trivial.
\end{proof}

Upon relaxing nilpotence to solvability, compactness is no longer necessary.

\begin{example}\label{ex:toract}
  Given the usual rotation action of the circle $\bS^1$ on $\bC\cong \bR^2$, the group $\bG:=\bC\rtimes \bS^1$ is linearly type-I-preserving. This will be an immediate consequence of \Cref{th:solv}.
\end{example}

Rather, for solvable groups it is condition \Cref{item:23} of \Cref{pr:nilpcon} that is relevant to linear type-I preservation. 

\begin{theorem}\label{th:solv}
  For a connected solvable Lie group $\bG$ the following conditions are equivalent.
  \begin{enumerate}[(a)]
  \item\label{item:24} $\bG$ is linearly type-I-preserving.
  \item\label{item:25} $\bG$ does not surject onto $\bR$.
  \item\label{item:26} The abelianization
    \begin{equation*}
      \bG_{ab}:=\bG/\overline{[\bG,\bG]}
    \end{equation*}
    is compact.
  \item\label{item:27} $\bG_{ab}$ is a torus.
  \item\label{item:28} $\bG$ is of the form $\bM\rtimes \bK$, where
    \begin{itemize}
    \item $\bK\cong \bT^n$ is a torus;
    \item and $\bM$ is a nilpotent, connected Lie group;
    \item so that the abelianization $\bM_{ab}\cong \bR^d$ is a vector group carrying a $\bK$-action with no trivial summands.
    \end{itemize}
  \item\label{item:29} $\bG$ has a closed, connected, cocompact subgroup which acts by unipotent operators in every finite-dimensional $\bG$-representation.
  \end{enumerate}
\end{theorem}
\begin{proof}
  Deducing \Cref{item:28} from the other conditions is the most effortful part of the proof, so we defer it and address everything else first.

  {\bf \Cref{item:24} $\Rightarrow$ \Cref{item:25}.} Immediate from \Cref{le:quot}, recalling that $\bR$ is not linearly type-I-preserving because of Mautner's example \cite[\S 6.8, 1.]{folland}.

  {\bf \Cref{item:25}, \Cref{item:26} and \Cref{item:27} are all equivalent.} The abelianization $\bG_{ab}$ is a connected abelian Lie group, and hence a product $\bT^n\times \bR^d$ of a torus and a vector group \cite[Theorem 4.2.4]{de}. Clearly, then, failure to surject onto $\bR$ is equivalent to compactness and to being a torus.

  {\bf \Cref{item:28} $\Rightarrow$ \Cref{item:29}.} The $\bM$ of \Cref{item:27} is precisely such a group. It is cocompact by assumption, so it remains to argue the unipotence claim. Let
  \begin{equation*}
    \pi:\bG\to GL(V)
  \end{equation*}
  be a finite-dimensional representation. Because $\bM$ is solvable and connected we can assume \cite[Theorem V.5.1$^*$]{serre-lag}, after complexifying, that $\pi(\bM)$ leaves invariant a complete flag
  \begin{equation*}
    \{0\}=V_0\subset V_1\subset\cdots\subset V_{\dim V}=V_{\bC}:=V\otimes_{\bR}\bC
  \end{equation*}
  with one-dimensional (complex) subquotients $V_{i+1}/V_i$. Each subquotient corresponds to a character
  \begin{equation*}
    \bM\to \bM_{ab}\to \bC^{\times};
  \end{equation*}
  since everything in sight carries a $\bK$-action, the multiset of such characters must be invariant under that action. But by assumption the non-trivial orbits of the $\bK$-action on $\bM_{ab}$ (and hence also on its space of characters) are all infinite. It follows, then, that $\bM$ acts (via $\pi$) trivially on each $V_{i+1}/V_i$, i.e. by unipotent operators on $V$ as a whole.
  
  {\bf \Cref{item:29} $\Rightarrow$ \Cref{item:24}.} Let $\bM\le \bG$ be a cocompact subgroup with the requisite unipotence property, and $\pi:\bG\to GL(V)$ a finite-dimensional representation. The fact that $\pi(\bM)$ consists of unipotent operators on $V$ implies that $V\rtimes \bM$ is a nilpotent, connected Lie group, hence type-I \cite[Theorem 7.8 (b)]{folland}. But then so is the larger group in the cocompact embedding
  \begin{equation*}
    V\rtimes \bM\le V\rtimes \bG,
  \end{equation*}
  by \Cref{cor:ext}. 
  
  Finally, one last implication will complete the circle.
  
  {\bf \Cref{item:27} $\Rightarrow$ \Cref{item:28}.} Several claims need proving.

  {\bf Step (1): The abelianization of $\overline{[\bG,\bG]}$ is a vector group.} Consider the metabelian (i.e. abelian-by-abelian) quotient
  \begin{equation*}
    \{1\}\to \bA\to \bH\to \bG_{ab}\to \{1\},
  \end{equation*}
  of $\bG$, where $\bA$ is the abelianization $\overline{[\bG,\bG]}_{ab}$ of the derived group (so that $\bH$ is the quotient of $\bG$ by its {\it second} derived group). Since $\bA$ is connected and abelian, it decomposes as
  \begin{equation}\label{eq:art}
    \bA\cong \bR^d\times \bT^m
  \end{equation}
  (\cite[Theorem 4.2.4]{de} again). The conjugation action of the torus $\bG_{ab}$ on $\bA$ will leave its unique maximal compact abelian subgroup $\bT^m$ invariant, and hence centralize it because $\bG_{ab}$ is connected and
  \begin{equation*}
    \mathrm{Aut}(\bT^m)\cong GL(\bZ,m)
  \end{equation*}
  is discrete. Further, because $\bG_{ab}$ is compact, its adjoint action on the Lie algebra $\mathrm{Lie}(\bA)$ is completely reducible, so there is a $\bG_{ab}$-invariant complement to the invariant subspace
  \begin{equation*}
    \mathrm{Lie}(\bT^m)\le \mathrm{Lie}(\bA);
  \end{equation*}
  we can thus assume that the decomposition \Cref{eq:art} is $\bG_{ab}$-invariant, so we can quotient by the $\bR^d$ factor to obtain a compact quotient
  \begin{equation*}
    \{1\}\to \bT^d\to \bullet\to \bG_{ab}\to \{1\}
  \end{equation*}
  of $\bG$. Since compact solvable groups are abelian \cite[Corollary V.5.3$^*$]{serre-lag} and we are assuming $\bG_{ab}$ is the {\it largest} abelian quotient of $\bG$, we must have $d=0$. This concludes the proof of Step (1).

  {\bf Step (2): The $\bG_{ab}$-action on $\overline{[\bG,\bG]}_{ab}$ has no trivial summands.} Let $\bR^d$ be such a summand. Quotienting by a $\bK$-representation complementary to $\bR^d$ produces a central extension
  \begin{equation*}
    \{1\}\to \bR^d\to \bullet\to \bK\to \{1\},
  \end{equation*}
  which splits \cite[Theorem 2.3]{moore-ext} as a product $\bR^d\times \bK$. Since that product is obtainable as a quotient of $\bG$, the compactness assumption on $\bG_{ab}$ ensures that $d$ vanishes.
  
  {\bf Step (3): $\bG\to \bG_{ab}$ splits.} We prove this by induction on the length of the (closed) derived series
  \begin{equation*}
    \{1\}\subset \cdots \subset \overline{[[\bG,\bG],[\bG,\bG]]} \subset \overline{[\bG,\bG]}\subset \bG,
  \end{equation*}
  the base case of abelian groups being trivial.

  For the induction step, fit $\bG$ into the extension
  \begin{equation*}
    \{1\}\to \bA\to \bG\to \bH\to \{1\},
  \end{equation*}
  where $\bA$ is the smallest (abelian) term of the derived series. We have the splitting claim for $\bH$ by the induction hypothesis, so we can pass from $\bG$ to its subgroup
  \begin{equation}\label{eq:abulletk}
    \{1\}\to \bA\to \bullet\to \bK\to \{1\}
  \end{equation}
  for a torus $\bK\le \bH$ surjecting onto $\bH_{ab}\cong \bG_{ab}$. As in the argument for Step (1), we have a $\bK$-invariant decomposition \Cref{eq:art}; this means that the extension \Cref{eq:abulletk} corresponds to a cohomology class
  \begin{equation*}
    \alpha\in H^2(\bK,\bR^d\times \bT^m)\cong H^2(\bK,\bR^d)\oplus H^2(\bK,\bT^m),
  \end{equation*}
  and note that both of these two latter cohomology groups vanish:
  \begin{itemize}
  \item the first by \cite[Theorem 2.3]{moore-ext} because $\bK$ is compact;
  \item and the second because dualizing an (automatically-abelian \cite[Corollary V.5.3$^*$]{serre-lag}) extension of a torus by a torus gives an abelian extension of a free abelian group by another, and free abelian groups are projective in the category of abelian groups \cite[Theorem 3.5]{rot}.
  \end{itemize}
  \Cref{eq:abulletk} then splits, and we are done. 

  {\bf Step (4): Conclusion.} The pieces are ready for assembly: take the closed derived subgroup $\overline{[\bG,\bG]}$ for $\bM$, and let $\bK\le \bG$ be the image of any splitting $\iota:\bG_{ab}\to \bG$. 
\end{proof}



\addcontentsline{toc}{section}{References}

\Addresses

\end{document}